\documentclass[a4,11pt]{article}

\usepackage{amssymb,amsfonts,amsmath,latexsym,color, amscd} 
\usepackage[all,cmtip]{xy}

\usepackage[dvips]{psfrag,graphicx} 
\usepackage{cancel}

\newtheorem{theorem}{Theorem}[section]
\newtheorem{lemma}[theorem]{Lemma}
\newtheorem{corollary}[theorem]{Corollary}
\newtheorem{proposition}[theorem]{Proposition}

\newtheorem{definition}[theorem]{Definition}

\newenvironment{proof}{{\par\addvspace{0.1cm}\noindent \bf Proof. }}{\hfill$\Box$\par\medskip}

\newtheorem{remark}[theorem]{Remark}

\numberwithin{equation}{section}

\def\e{\varepsilon}
\def\w{\wedge}

\def\R{\Re\mathfrak{e} \,}%\def\R{\Re\mbox{{\rm e}} \,}
\def\I{\Im\mathfrak{m} \,}%\def\I{\Im\mbox{{\rm m}} \,}
\def\omegasub#1{\mbox{\large $\omega$}\mbox{\small ${\mbox{\large ${}$}}_{#1}$}}% 2-form omega: enlarge omega and make suffix smaller
\def\vect#1{\mbox{\boldmath $#1$}} % to make symbols in boldmath to express vectors 
\def\wt#1{\widetilde{#1}}
\def\da#1{d^2{#1}}
\def\T{f} %<- 02/10/2010
\begin{document}

\title{M\"obius invariant energies and average linking with circles}

\author{Jun O'Hara and Gil Solanes}
%
%\date{}
%
\maketitle

\begin{abstract} We define and study a M\"obius invariant energy associated to planar domains, as well its generalization to space curves. This generalization is a M\"obius version of Banchoff-Pohl's notion of area enclosed by a space curve. A relation with Gauss-Bonnet theorems for complete surfaces in hyperbolic space is also described.
\end{abstract}



\section{Introduction} 
For a closed smooth curve $K\subset \mathbb R^3$, the {\em M\"obius energy} was defined in \cite{jun} as the integral on $K$ of the following function called {\em renormalized potential}
\begin{equation}\label{potencial}
V(p,K)=\lim_{\e\to 0}\left(\int_{|q-p|>\e}\frac{dq}{|q-p|^2}-\frac{2}{\e}\right),\qquad p\in K.
\end{equation}
The resulting functional is invariant under the action of the M\"obius group as shown first in \cite{bryson}. A generalization to  surfaces was carried out by Auckly and Sadun in \cite{auckly}. For a point $p$ in a compact embedded surface $\Omega\subset\mathbb R^n$ they defined a renormalized potential $V(p,\Omega)$ analogous to \eqref{potencial}. 
When the surface $\Omega$ has empty boundary, the integral of $V(p,\Omega)$ on $\Omega$ yields a M\"obius invariant energy. However, when $\Omega$ has non-empty boundary, $V(p,\Omega)$ blows up near $\partial \Omega$,  causing the divergence of the integral of $V(p,\Omega)$. In the present paper we consider a renormalization of this integral in  case $\Omega$ is a planar domain with non-empty smooth boundary.  
This defines a M\"obius invariant energy for planar domains. This energy turns out to be related to a recent Gauss-Bonnet formula for complete surfaces in hyperbolic space (cf. \cite{Gil}).

When the energy of a planar domain is expressed {by the contour integral} on its boundary, a generalization to space curves appears naturally. This is a renormalization of the M\"obius invariant measure of the set of circles linked with the curve. Again, this functional appears in a Gauss-Bonnet formula for surfaces in hyperbolic space. Besides connections to knot energies, and hyperbolic geometry, our results may be interesting from the viewpoint of integral geometry. Indeed, due to divergence problems, almost nothing is known about integral geometry under the M\"obius group (an exception is \cite{La-OH}). Here, the use of renormalization allows us to extend some results of euclidean integral geometry to M\"obius geometry. 
 In fact, our functional for space curves can be seen as a M\"obius invariant version of Banchoff-Pohl's notion of the area enclosed by a space curve (cf. \cite{banchoff.pohl}).

\bigskip
Next we sketch our results briefly. For a planar region $\Omega\subset\mathbb R^2$ Auckly-Sadun's renormalized potential is given  by (cf. Definition \ref{defvoltage})
\[
 V(w,\Omega)=
\displaystyle \lim_{\e\to 0}\left(\>\int_{\Omega\setminus B_\e(w)}\frac{\da z}{|z-w|^4}-\frac{\pi}{\e^2}\>\right),
\]
{where $\da z$ is} the area element of $\mathbb R^2$. The counterterm $\frac{\pi}{\e^2}$ is chosen in order to cancel the blow-up of the integral as {$\e$ goes down to $0$}. After studying the blow-up of $V(\cdot,\Omega)$ itself near $\partial \Omega$  (cf. Proposition \ref{lem_V_Laurent_e})  we define the energy $E(\Omega)$ of a region $\Omega$ bounded by a smooth curve $K$ of length $L(K)$ as 
\begin{equation}\label{defdomainintro}
 E(\Omega)=\lim_{\delta\to 0}\left(\int_{\Omega_\delta}V(w,\Omega)\da w+\frac{\pi}{4\delta}L(K)\right),
\end{equation}
where $\Omega_\delta\subset\Omega$ is the  set of points at distance bigger than $\delta$ from $\Omega^c=\mathbb R^2\setminus \Omega$. This is a renormalization of the integral of $V(\cdot,\Omega)$. Alternatively, given  a smoothly embedded curve $K\subset\mathbb R^2$ we define (cf. Definition \ref{defcurve})
\begin{equation}\label{defcurvintro}
   E(K)=\lim_{\e \to 0}\left(\frac{2L(K)}{\e}-\int_{\Omega\times\Omega^c\setminus\Delta_\e}\frac{\da w\da z}{|z-w|^4}\right),
\end{equation}where $\Delta_\e\subset\mathbb R^2\times\mathbb R^2$ {consists of} pairs $(w,z)$ with $|z-w|<\e$.
 For $\Omega\subset\mathbb R^2$ compact and $K=\partial \Omega$ both energies are related by $E(\Omega)=E(K)+\pi^2\chi(\Omega)/4$ (cf. Proposition \ref{relation}). Among several expressions for these energies we point {out} the following one which involves no limit:
\[
E(K)=-\frac12\int_{K\times K}\sin\theta_p\sin\theta_q\frac{dpdq}{|q-p|^2},
\] where $dp,dq$ denote {the} arc-length elements, and $\theta_p$ (resp. $\theta_q$) is the angle between  $q-p$ and $K$ at $p$ (resp. at $q$).

Considering $\mathbb R^2$ as the ideal boundary of Poincar\'e half-space model of hyperbolic space $\mathbb H^3$, we can assume $K$ to be the ideal boundary of a smooth surface $S\subset\mathbb H^3$ meeting $\mathbb R^2$ orthogonally. Then we have the following Gauss-Bonnet formula (cf. Proposition \ref{gbtres})
\[
 \int_S \kappa\ dS=\ 2\pi\chi(S) +\frac{2}{\pi}\int_{\mathbb R^2\times\mathbb R^2}(\#(\ell_{wz}\cap S)-\lambda^2(w,z;K))\frac{\da w\da z}{|z-w|^4}-\frac{4}{\pi}E(K), 
\]
where $\kappa$ denotes the extrinsic curvature of $S$, and $\ell_{wz}$ denotes the geodesic with ideal endpoints $w,z$, while $\lambda(w,z;K)$ is the algebraic intersection {number} of $K$ with the segment $[zw]\subset\mathbb R^2$.
As a consequence, we get the M\"obius invariance of $E(K)$ and $E(\Omega)$ (cf. Corollary \ref{invariancia}), as well as some lower bounds (cf. Corollary \ref{bounds}).

\bigskip
For a closed curve $K\subset\mathbb R^3$ we define $E(K)$ as the renormalized measure of the set of circles linked with $K$. Indeed, there is a natrual (M\"obius invariant) measure $d\gamma$ on $\mathcal S(1,3)$, the space of oriented circles $\gamma\subset \mathbb R^3$. To be precise we define
\[
  E(K)=\lim_{\e\to 0}\left(\frac{3\pi L(K)}{8\e}-\frac{3}{16\pi}\int_{\mathcal S_\e(1,3)}\lambda^2(\gamma,K)d\gamma\right)
\]
where $\mathcal S_\e(1,3)$ is the set of oriented circles with {radii} bigger than $\e$, and $\lambda(\gamma,K)$ denotes the linking number between $\gamma$ and $K$. This is motivated by \eqref{defcurvintro}, and indeed both defnitions coincide when $K$ is planar. Again, we find an expression of $E(K)$ that involves no renormalization:
\[
 E(K)=-\frac12\int_{K\times
K
}{\cos\tau\sin\theta_p\sin\theta_q}\frac{d pd q}{|q-p|^2},
\] where $\tau$ is the angle between the two oriented planes through $p,q$ tangent to $K$ at $p$ and $q$ respectively. It is interesting to remark that replacing $\cos\tau$ by $\sin\tau$ gives the so-called {\em writhe} of $K$, another M\"obius invariant functional for space curves. Besides, if the power in the denominator is replaced by 1 or 0, one gets respectively the length of $K$ and Banchoff-Pohl's area enclosed by $K$.

Again, $E(K)$ appears in a Gauss-Bonnet formula:  if a surface $S\subset \mathbb H^4$ in Poincar\'e half-space model of 4-dimensional hyperbolic space meets the ideal boundary orthogonally along a closed curve $K$, then (cf. Corollary \ref{gb})
\[
 \frac1\pi\int_{N^1S} \kappa \ de\,dS=2\pi\chi(S)+\frac{3}{4\pi^2}\int_{\mathcal L_{2}^+}(\#(\ell\cap S)-\lambda^2(\ell,K))d\ell-\frac{2}{\pi}E(K),
\]
where $\kappa$ denotes the Lipschitz-Killing curvature defined on the unit normal bundle $N^1S$, {and $de$ denotes the volume element of $N_x^1S$}.  The second integral takes place on the space $\mathcal L_2^+$ of oriented totally geodesic planes $\ell\subset\mathbb H^4$, which is naturally identified to $\mathcal S(1,3)$, and the measure $d\ell$ corresponds to $d\gamma$. By construction, these two integrals are invariant under isometries of $\mathbb H^4$. This shows the M\"obius invariance of $E(K)$. 

A direct proof of this invariance, without use of hyperbolic space, is given at the end of the paper. %\jun{(We may remove this direct proof)}
In Proposition \ref{prop_cos_cos_K_Kepsilon_space} we provide an alternative construction of $E(K)$ inspired by \eqref{defdomainintro}. There, we approach $K$ by a parallel curve $K_\delta$, and we integrate the  product of linking numbers $\lambda(\gamma,K)\lambda(\gamma,K_\delta)$ over all circles $\gamma\in\mathcal S(1,3)$. In section \ref{invariant.expressions} we go back to the planar case and give some M\"obius invariant expressions of the energy of a domain.

\medskip
Acknowledgement: The authors would like to thank Professor M. Kanai for helpful suggestions. 

%%%%%%%%%%%%%%%%%%%%%%%%%%%%%%%%%%%%%%%%%%%%%%%%%%%%%%%%%%%%%%%%%%%%%%%%%%%%%%%%%%%%%%%%%%%

\section{Infinitesimal cross ratio}
We start fixing some notations, and presenting some tools that will be used along the paper.
We shall be considering pairs of complex  numbers $w=u+iv, z=x+iy\in \mathbb C$. We denote the diagonal in $\mathbb C\times\mathbb C$ by $\Delta=\{(w,w)\}$. 
The {\em infinitesimal cross ratio} (\cite{La-OH}) is a complex valued $2$-form on $\mathbb C\times\mathbb C\setminus\Delta$ given by \[\omega_{cr}=\frac{dw\w dz}{(w-z)^2}=\frac{(du+idv)\w (dx+idy)}{(w-z)^2}\,.\] 
It is invariant under diagonal action of (orientation preserving)  M\"obius transformations: $h(z)=(az+b)/(cz+d)$ where $a,b,c,d\in\mathbb C$ and $ad-bc\neq 0$. Recall that such an $h$ defines a transformation $h:\mathbb{CP}^1\rightarrow  \mathbb{CP}^1$ of the Riemann sphere $\mathbb{CP}^1=\mathbb C\cup\{\infty\}$. For simplicity we will work with $\mathbb C$ instead of $\mathbb{CP}^1$. This causes no trouble, except that $h$ is not defined in one point.  

Both the real part and the imaginary part of the infinitesimal cross ratio are exact forms; 
\begin{equation}
d\left(\R\frac{dw}{w-z}\right)=d\left(\R\frac{dz}{z-w}\right)=-\R \omega_{cr}, \quad d\left(\I\frac{dw}{w-z}\right)=d\left(\I\frac{dz}{z-w}\right)=-\I \omega_{cr}\,.
\end{equation}
Direct computation shows 
\begin{equation}\label{squares}
 \R\omega_{cr}\w\R\omega_{cr}=\I\omega_{cr}\w\I\omega_{cr}=2\frac{\da w\w \da z}{|z-w|^4},
\end{equation}
where $\da w=du\w dv,\da z=dx\w dy$ are {the} area elements in $\mathbb C$. At some places we will omit the wedges in the exterior product of forms when these are understood as measures. Note that the forms in \eqref{squares} are invariant under all M\"obius transformations, preserving orientation or not.
\subsection{Interpretation via hyperbolic space}
It will be sometimes useful to consider $\mathbb C$ as the ideal boundary of Poincar\'e half-space model of hyperbolic space. The reason behind is that hyperbolic motions induce M\"obius transformations on the boundary. Given $(w,z)\in \mathbb C\times\mathbb C\setminus\Delta$ we can consider the oriented geodesic $\mathbb \ell_{wz}$ with ideal endpoints $w,z$ at $-\infty,+\infty$ respectively. Let us choose (locally) for each pair $(z,w)$ a point $o\in \ell_{wz}\subset\mathbb H^3$ and an oriented orthonormal frame $e_1,e_2,e_3\in T_o\mathbb H^3$ with respect to the hyperbolic metric $\langle\ ,\ \rangle$. Then the differential 1-forms $\omega_i=\langle do,e_i\rangle$ are (locally) defined in $\mathbb C\times\mathbb C\setminus\Delta$. Similarly, we have the connection forms $\omega_{ij}=\langle \nabla e_i,e_j\rangle$ where $\nabla$ denotes the riemannian connection of $\mathbb H^3$. It turns out that $d\omega_1$ is independent of the choice of the point $o$ and hence of the frame. Indeed, if $o'=\exp_o(f\cdot e_1)$ is a second choice, one gets $\omega_1'= \omega_1+df$. Similarly, $d\omega_{23}$ is independent of the choice of the frame (cf. \cite[Prop.5]{Gil}). Hence, $d\omega_1$, $d\omega_{23}$ are well-defined global forms on $\mathbb C\times\mathbb C\setminus\Delta$, invariant under orientation preserving M\"obius transformations. The structure equations of hyperbolic space yield
\[
 d\omega_1=\omega_{12}\w\omega_2+\omega_{13}\w\omega_3,\qquad d\omega_{23}=\omega_2\w\omega_3-\omega_{12}\w\omega_{13}.
\]
Let us take $o=(\frac{u+x}{2}, \frac{v+y}{2}, r)\in\ell_{wz}$, {where} $r=|z-w|/2$ and $w=u+iv,z=x+iy$ as before; i.e. $o$ maximizes the third coordinate in $\ell_{wz}$. After a horizontal dispalcement we can assume $v=y=0$ and $-u=x=r$. Then we can choose the frame $e_1=(r,0,0), e_2=(0,r,0), e_3=(0,0,r)\in T_o\mathbb H^3$. Then,
\[
 \omega_2=\frac{1}{2r}(dy+dv),\qquad\omega_3=\frac{1}{2r}(dx-du),\qquad \omega_{12}=\frac{1}{2r}(dy-dv),\qquad\omega_{13}=\frac{1}{2r}(dx+du).
\]
 Therefore,
\[
 d\omega_1=2\R\omega_{cr},\qquad d\omega_{23}=2\I\omega_{cr}.
\]

\subsection{Symplectic forms on Grassmannians} 

By taking the hyperboloid model of $\mathbb H^3$, the ideal boundary is identified to the set of lines in the light cone of the Minkowski space $\mathbb R^4_1$. %} 
Thus the space $\mathbb C\times\mathbb C\setminus\Delta$ can be identified with %\gil
{a dense open set of} a Grassmannian manifold $SO(3,1)/SO(2)\times SO(1,1)$ which is the space of the oriented timelike $2$-planes in $\mathbb R^4_1$. Two kinds of interpretation of the space, one as the cotangent bundle $T^{\ast}\mathbb R^2$ and the other as a K\"ahler manifold reveal the meanings of the real and the imaginary parts of the infinitesimal cross ratio. 

\subsubsection{The real part of {\boldmath $\omega_{cr}$} as a canonical symplectic form of {\boldmath $T^{\ast}\mathbb R^2$}}
In this subsubsection we introduce some result from \cite{La-OH}. 

Let $\Psi:\mathbb S^n\times\mathbb S^n\setminus\Delta\to T^{\ast}\mathbb S^n$ be the bijection given by $\Psi(x,y)=(x,\Psi_x(y))$, where $\Psi_x:\mathbb S^n\setminus \{x\}\rightarrow (x)^\bot\equiv T_x^*\mathbb S^n$ is the stereographic projection. If $P:\mathbb R^n\rightarrow \mathbb S^n$ is the inverse of the stereographic projection, then $\psi=P^*\circ\Psi\circ (P\times P)$ is a bijection between $\mathbb R^n\times\mathbb R^n\setminus\Delta$ and $T^{*}\mathbb R^n$.%}
\par\noindent
The pull-back of the canonical symplectic form $\omegasub{T^{\ast}\mathbb R^n}$ of the cotangent bundle $T^{\ast}\mathbb R^n$ by $\psi$ is given by 
\begin{eqnarray}\label{symplectic}
\psi^\ast\omegasub{T^{\ast}\mathbb R^n}&=&\displaystyle 
-2d\left(\frac{\sum_{i=1}^{n} (z_i-w_i)dw_i}{|z-w|^2}\right)
=-2d\left(\frac{\sum_{i=1}^{n} (z_i-w_i)dz_i}{|z-w|^2}\right)\label{eq_omega_cotan_one-form}\\%[4mm]
&=&\displaystyle{2\left(\frac{\sum_{i=1}^{n} dw_i\w dz_i}{|z-w|^2}
-2\frac{(\sum_{i=1}^{n} (z_i-w_i) dw_i)\w 
(\sum_{j=1}^{n} (z_j-w_j)dz_j)}{|z-w|^4}\right).} \nonumber
\end{eqnarray}

\noindent
We will hereafter identify $\mathbb R^n\times\mathbb R^n\setminus\Delta$ with $T^{\ast}\mathbb R^n$ through $\psi$ and denote $\psi^\ast\omegasub{T^{\ast}\mathbb R^n}$ simply by $\omegasub{T^{\ast}\mathbb R^n}$. 

Especially, when $n=2$ then we have (by folklore) 
\begin{equation}\label{eq_folklore}
\R\omega_{cr}=-\frac12\omegasub{T^{\ast}\mathbb R^2}.
\end{equation}

This $2$-form $\omegasub{T^{\ast}\mathbb R^n}$ is invariant under the diagonal action of a M\"obius transformation $\T$, i.e. $(\T\times \T)^\ast\omegasub{T^{\ast}\mathbb R^n}=\omegasub{T^{\ast}\mathbb R^n}$, although the $1$-forms $\frac{\sum_{i=1}^{n} (z_i-w_i)dw_i}{|z-w|^2}$ and $\frac{\sum_{i=1}^{n} (z_i-w_i)dz_i}{|z-w|^2}$ are not. 
We remark that our bijection $\psi$ is not compatible with the diagonal action of a M\"obius transformation on $\mathbb R^n\times\mathbb R^n\setminus\Delta$, i.e. $\psi\circ (\T\times \T)\neq (\T^*)^{-1}\circ\psi$. 

As before, we can consider $\mathbb R^n$ as the boundary of half-space model $\mathbb H^{n+1}$. Then $(w,z)\in \mathbb R^n\times\mathbb R^n\setminus\Delta$ are the ideal endpoints of a geodesic $\ell$. Again, given $o\in\ell$ and a unit vector $e_1 \in T_o\ell$, we get a $1$-form $\omega_1=\langle do,e_1\rangle$, and $d\omega_1$ is independent of the chosen $o,e_1$. By taking  $o=\frac12(z+w, |z-w|)\in \mathbb H^{n+1}$ and $e_1=\frac12(z-w,0)\in T_o\mathbb H^{n+1}$, it is easy to check from \eqref{symplectic} that $\omega_{T^*\mathbb R^n}=-d\omega_1$.

\subsubsection{The imaginary part of {\boldmath $\omega_{cr}$} as a K\"ahler form}
Let $\mathcal{S}(n-2,n)$ be the set of oriented codimension $2$ subspheres in $\mathbb S^n$. 
We can realize $\mathbb S^n$ in the Minkowski space $\mathbb R^{n+2}_1$ as the intersection of the light cone and a spacelike affine hyperplane. 
Therefore $\mathcal{S}(n-2,n)$ can be identified with the set of oriented timelike codimension $2$ subspaces of $\mathbb R^{n+2}_1$. 
Let us denote it by $G$. 
It is a non-compact Grassmannian manifold $SO(n+1,1)/SO(2)\times SO(n-1,1)$ with an indefinite pseudo inner product $\langle \>,\>\rangle$. 
Just like in compact case, $G$ has a K\"ahler form $\omegasub{K}$ defined by $\omegasub{K}(u,v)=\langle Ju, v\rangle$ $(u,v\in T_\Pi G, \Pi\in G)$, where $J$ is the complex structure given by a $90^\circ$ degrees rotation which can be considered as an element of $SO(2)$. This K\"ahler form $\omegasub{K}$ is invariant under orientation preserving M\"obius transformations on $\mathbb S^n$, which are given by the action by elements of $SO(n+1,1)$ on the light cone in $\mathbb R^{n+2}_1$. 

\begin{proposition}
When $n=2$, 
\[
\I\omega_{cr}=-\frac12\omegasub{K}.
\]
To be precise, the right hand side should be understood to be $-\frac12(\T\times \T)^\ast\omegasub{K}$, where $\T:\mathbb R\to\mathbb S^2\setminus\{\textrm{\rm pt.}\}$ is the inverse of an orientation preserving stereographic projection. 
\end{proposition}
\begin{proof}
As both $\omega_{cr}$ and $\omegasub{K}$ are invariant under M\"obius transformations, we may fix a point $\Pi$ in $G$. Suppose $e_0, e_1, e_2, e_3$ form a pseudo-orthonormal basis of $\mathbb R^4_1$ with $e_0\cdot e_0=-1 $ and $e_i\cdot e_j=\delta_{ij} \,((i,j)\ne(0,0))$. Assume $\Pi=\textrm{Span}\langle e_0, e_1\rangle$. Then $T_\Pi G\cong \textrm{Hom}(\Pi,\Pi^{\perp})$ is spanned by $v_{ij}$ $(i=0,1,\,j=2,3)$, where $v_{ij}\in\textrm{Hom}(\Pi,\Pi^{\perp})$ is given by $v_{ij}(e_i)=e_j$ and $v_{ij}(e_{1-i})=0$. They form a pseudo-orthonormal basis of $T_\Pi G$ with $\langle v_{0j},v_{0j}\rangle=-1$ and $\langle v_{1j},v_{1j}\rangle=1$ $(j=2,3)$. 

Since the complex structure $J$ is obtained by $90^\circ$ degrees rotation in the $e_2e_3$-plane, namely, $J(v_{i2})=v_{i3}$ $(i=0,1)$, 
we have $\omegasub{K}(v_{02}, v_{03})=-1$, $\omegasub{K}(v_{12}, v_{13})=1$, and $\omegasub{K}(v_{ij}, v_{kl})=0$ if $\{v_{ij}, v_{kl}\}$ is not equal to $\{v_{02}, v_{03}\}$ or $\{v_{12}, v_{13}\}$. 

On the other hand, by a suitable identification, $\Pi$ correspnds to $((u,v),(x,y))=((1,0),(-1,0))$ in $\mathbb R^2\times\mathbb R^2\setminus\Delta$ and $v_{ij}$ correspond to 
\[
v_{02}=\frac{\partial}{\partial v}+\frac{\partial}{\partial y}, \> v_{03}=-\frac{\partial}{\partial u}+\frac{\partial}{\partial x}, \>
v_{12}=\frac{\partial}{\partial v}-\frac{\partial}{\partial y}, \> v_{13}=-\frac{\partial}{\partial u}-\frac{\partial}{\partial x}. 
\]
Take care not to use a stereographic projection form the north pole here as it is orientation reversing. 
Now the direct computation shows that $\omegasub{K}=-2\I\omega_{cr}$. 
\end{proof}

%%%%%%%%%%%%%%%%%%%%%%%%%%%%%%%%%%%%%%%%%%%%%%%%%%%%%%%%%%%%%%%%%%%%%%%%%%%%%%%%%%%%%%%%%%%

\section{The M\"obius energy of pairs of disjoint planar domains}
Let $\Omega_1,\Omega_2$ be a pair of disjoint domains in $\mathbb R^2$ with smooth regular boundaries. Suppose {each pair of particles in $\Omega_1$ and $\Omega_2$ has a mutual repelling force between them}. Assume this force has magnitude $r^{-5}$ where $r$ denotes the distance between the particles. The reason for this exponent will be clear below. Under these assumptions the corresponding energy for the interaction of $\Omega_1$ and $\Omega_2$ would be the following.

\begin{definition} \rm \label{mutual}
The M\"obius mutual energy between $\Omega_1$ and $\Omega_2$ is defined as 
\[E(\Omega_1,\Omega_2)= \int_{\Omega_1\times\Omega_2}\frac{\da w \da z}{|z-w|^4}\,,\]
where $\da w$ (resp $\da z$) denotes the area element in $\Omega_1\subset\mathbb R^2$ (resp. $\Omega_2\subset\mathbb R^2$).
\end{definition}
This energy is invariant under M\"obius transformations. Indeed, \eqref{squares} implies 
\[E(\Omega_1,\Omega_2)
=\frac{1}{2}\int_{\Omega_1\times\Omega_2}\R\omega_{cr}\w\R\omega_{cr}
= \frac{1}{2}\int_{\Omega_1\times\Omega_2}\I\omega_{cr}\w\I\omega_{cr}\,.\]
%
% \subsection{Expression by double contour integral}
%
\begin{proposition}\label{first}
 Let $\Omega_1,\Omega_2\subset \mathbb R^2$ be a pair of disjoint planar domains with smooth regular boundaries $K_1=\partial \Omega_1$, $K_2=\partial \Omega_2$. 
Then $E(\Omega_1,\Omega_2)$ can be expressed by double contour integral: 
\begin{eqnarray}\label{rere}
E(\Omega_1,\Omega_2)%&=&\displaystyle \frac{1}{2}\int_{\Omega_1\times\Omega_2}\R\omega_{cr}\w\R\omega_{cr}\notag\\[2mm]
&=&\displaystyle-\frac{1}{2}\int_{K_1\times K_2}\cos\theta_1\cos\theta_2\frac{dp_1 dp_2}{|p_2-p_1|^2}\,, \quad{}
\end{eqnarray}
\begin{eqnarray}\label{imim}
E(\Omega_1,\Omega_2)%&=&\displaystyle \frac{1}{2}\int_{\Omega_1\times\Omega_2}\I\omega_{cr}\w\I\omega_{cr}\notag\\[2mm]
&=&-\frac{1}{2}\int_{K_1\times K_2}\sin\theta_1\sin\theta_2\frac{dp_1 dp_2}{|p_2-p_1|^2}\,, \quad{}
\end{eqnarray}
where $dp_i$ is the length element  on $K_i$, and $\theta_i$ is the oriented angle  from the positive tangent of $K_i$ at $p_i$ to the vector $p_2-p_1$.
\end{proposition}

\begin{proof}
Put 
\begin{equation} \label{abbreviation}
 \lambda=-\R\frac{dw}{w-z},\qquad \rho=-\R\frac{dz}{z-w},\qquad\omega=\R \omega_{cr}\,,
\end{equation}
so that $d\lambda=d\rho=\omega$. By Stokes' theorem
\[
\int_{\Omega_1\times\Omega_2}\omega\w\omega=\int_{(K_1\times \Omega_2)\cup(\Omega_1\times K_2)}\lambda\w\omega=\int_{K_1\times \Omega_2}\lambda\w\omega.
\]
Since $\lambda\wedge\omega=\omega\wedge\rho-d(\lambda\wedge\rho)$,
\[
\int_{\Omega_1\times\Omega_2}\omega\w\omega=\int_{K_1\times \Omega_2}\omega\w\rho-\int_{K_1\times K_2}\lambda\w\rho=-\int_{K_1\times K_2}\lambda\w\rho.
\]
{Then \eqref{rere} follows from elementary computations.} 
The same arguments with  $\R$ replaced by $\I$ in \eqref{abbreviation} proves \eqref{imim}. 
\end{proof}

\begin{corollary}\label{cor_33} %\jun{(I labeled this corollary)}
Under the above hypothesis
\begin{equation}\label{f_E12_contour}
E(\Omega_1,\Omega_2)
=-\frac14\int_{K_1\times K_2}\frac{\overrightarrow{dp_1}\cdot \overrightarrow{dp_2}}{|p_2-p_1|^2},
\end{equation}
where $\overrightarrow{dp_1}\cdot \overrightarrow{dp_2}=du_1\wedge du_2+dv_1\wedge dv_2$ is a $2$-form on $K_1\times K_2$ {where} $p_i=(u_i,v_i)\in K_i$.
\end{corollary}

\begin{proof}
Let $\phi_i$ ($i=1,2$) be the angle of a tangent vector to $\Omega_i$ from the $x$-axis. 
Then 
\[\cos\theta_1\cos\theta_2+\sin\theta_1\sin\theta_2=\cos(\theta_1-\theta_2)
=\cos(\phi_1-\phi_2)
=\cos\phi_1\cos\phi_2+\sin\phi_1\sin\phi_2.\]
Let $dp_i$ denote the length element of $K_i$. 
As $du_i=\cos\phi_i{dp_i}, {dv_i}=\sin\phi_i{dp_i}$ the above equation implies 
\begin{equation}\label{eq_cor33}
(\cos\theta_1\cos\theta_2+\sin\theta_1\sin\theta_2)dp_1dp_2
=\overrightarrow{dp_1}\cdot \overrightarrow{dp_2}\,.
\end{equation}
Therefore, by averaging (\ref{rere}) and (\ref{imim}) we have 
\begin{eqnarray*}
E(\Omega_1,\Omega_2)&=&-\frac14\int_{K_1\times K_2}
\frac{\cos\theta_1\cos\theta_2+\sin\theta_1\sin\theta_2}{|p_2-p_1|^2}\,dp_1dp_2%\\
=-\frac14\int_{K_1\times K_2}\frac{\overrightarrow{dp_1}\cdot \overrightarrow{dp_2}}{|p_2-p_1|^2}. 
\end{eqnarray*}

\end{proof}

%%%%%%%%%%%%%%%%%%%%%%%%%%%%%%%%%%%%%%%%%%%%%%%%%%%%%%%%%%%%%%%%%%%%%%%%%%%%%%%%%%%%%%%%%%%

\section{Renormalized M\"obius energy of planar domains}
Let $\Omega\subset \mathbb R^2$ be a planar domain with smooth boundary $K=\partial \Omega$. We will define a M\"obius invariant energy associated to $\Omega$. One cannot take $E(\Omega,\Omega)$ because of the blow up of $\omega_{cr}$ near the diagonal $\Delta\subset \Omega\times\Omega$. This kind of difficulty is usually avoided by means of the so-called renormalization (also called regularization) procedures. 
We introduce two kinds of renormalizations and show that they produce essentially the same energy. 
The second renormalization will appear later in subsection \ref{subsection_second_renormalization}.

The first renormalization consists of two steps. First we define a renormalized potential at every point of the domain. This is a particular case of the potential considered in \cite{auckly} for general surfaces. The integral of this potential is divergent when the boundary is not empty. Hence we need a second step where this integral is renormalized.

\subsection{Renormalized potential}
\begin{definition}[\cite{auckly,OH2}]\label{defvoltage} \rm 
Let $w$ be a point in $\Omega\setminus\partial\Omega$. 
We define the renormalized {$r^{-4}$-po\-ten\-tial} of $\Omega$ at $w$ by
\begin{equation}\label{voltage}
 V(w,\Omega)=
\displaystyle \lim_{\e\to 0}\left(\>\int_{\Omega\setminus B_\e(w)}\frac{\da z}{|z-w|^4}-\frac{\pi}{\e^2}\>\right) 
\end{equation}
\end{definition}

\begin{proposition}[\cite{OH2}]
The renormalized potential of $\Omega$ at an interior point $w$ is given by 
\begin{eqnarray}\label{f_V_Omega^c}
V(w,\Omega)=-\int_{\Omega^c}\frac{\da z}{|z-w|^4},
\end{eqnarray}
where $\Omega^c=\mathbb R^2\setminus\Omega$ denotes the complement of $\Omega$. 
Hence $-\infty<V(w,\Omega)<0$. 
\end{proposition}
\begin{proof} 
Since 
\[\int_{\mathbb R^2\setminus B_\e(w)}\frac{\da z}{|z-w|^4}=\frac{\pi}{\e^2}\,,\]
if $\e>0$ is such that $B_\e(w)\subset \Omega$ then 
\[\begin{array}{rcl}
\displaystyle \int_{\Omega\setminus B_\e(w)}\frac{\da z}{|z-w|^4}-\frac{\pi}{\e^2}
&=&\displaystyle \int_{\Omega\setminus B_\e(w)}\frac{\da z}{|z-w|^4}-\int_{\mathbb R^2\setminus B_\e(w)}\frac{\da z}{|z-w|^4}%\\[4mm]
=\displaystyle -\int_{\mathbb R^2\setminus \Omega}\frac{\da z}{|z-w|^4}\,.
\end{array}\]
\end{proof}

In view of \eqref{f_V_Omega^c} one can interpret $V(w,\Omega)$ as the area of the image of $\Omega^c$ after an inversion with respect to a circle of center $w$ and radius $1$. Indeed, the Jacobian of such an inversion is precisely  $-|w-z|^{-4}$.

\begin{proposition}[\cite{OH2}]
The renormalized potential $V(w,\Omega)$ can be expressed by a contour integration as 
\begin{equation}\label{V_contour_int}
\begin{array}{rcl}
V(w,\Omega)&=&\displaystyle \frac14\int_{K}\nabla\left(\frac 1{r(p)^2}\right)\cdot \vect n(p)\, dp\\[4mm]
&=&\displaystyle -\frac12\int_{K}\frac{(x-u)dy-(y-v)dx}{|p-w|^4}\,\\[4mm]
&=&\displaystyle -\frac12\int_{K}\frac{\textrm{\rm det}\big(p-w, \overrightarrow{dp}\,\big)}{|p-w|^4}\,, 
\end{array}
\end{equation}
where $r(p)=|p-w|$, $\vect n$ is the outer unit normal vector to $K$ and $w=(u,v)\in\Omega$, $p=(x,y)\in K$, and $\overrightarrow{dp}=(dx,dy)$. 
\end{proposition}

\begin{proof}
{Take a small positive number $\e$ so that $B_\e(w)\subset\Omega$.} 
Since 
\[\Delta r^{-2}=\left(\frac{\partial^2}{\partial x^2}+\frac{\partial^2}{\partial y^2}\right)
\{(x-u)^2+(y-v)^2\}^{-1}=4r^{-4}\,,
\]
we have
\begin{equation}\label{f_nabla1}
\frac14\int_{\partial(\Omega\setminus B_\e(w))}\nabla r^{-2}\cdot \vect n\, ds
=\frac14\int_{\Omega\setminus B_\e(w)}\Delta r^{-2}dxdy
=\int_{\Omega\setminus B_\e(w)}\frac{\da z}{|z-w|^4}\,.
\end{equation}
As the left hand side is equal to 
\[\begin{array}{l}
\displaystyle \frac14\int_{\partial\Omega}\nabla r^{-2}\cdot \vect n\, dp
-\frac14\int_{\partial B_\e(w)}\nabla r^{-2}\cdot \vect n\, ds\\[4mm]
=\displaystyle \frac14\int_{K}\nabla r^{-2}\cdot \vect n\, dp
+\frac12\int_{\partial B_\e(w)}\frac{(x-u)dy-(y-v)dx}{|p-w|^4} \\[4mm]
=\displaystyle \frac14\int_{K}\nabla r^{-2}\cdot \vect n\, dp
+\frac\pi{\e^2}\,,
\end{array}
\]
formula \eqref{f_nabla1} implies that 
\[
V(w,\Omega)=\int_{\Omega\setminus B_\e(w)}\frac{\da z}{|z-w|^4}-\frac{\pi}{\e^2}
=\frac14\int_{\partial\Omega}\nabla r^{-2}\cdot \vect n\, dp.
\]
\end{proof}
\begin{proposition}\label{lem_V_Laurent_e}
The behavior of $V(w,\Omega)$ as $w\in\,\stackrel\circ\Omega$ approaches the boundary from inside can be expressed in terms of $\delta=d(w,K)$ by 
\begin{equation}\label{V_w->K}
V(w,\Omega)=-\left(\frac{\pi}{4\delta^2}+\frac{\kappa\pi}{4\delta}\right)+O(1),
\end{equation}
where $\kappa$ is the curvature of $K$ a the closest point in $K$ to $w$. 
\end{proposition}
\begin{proof} 
{Fix} a positive constant $\e$ which is smaller than the minimum radius of curvature of $K$. 
Suppose $\delta=d(w,K)$ satisfies $\delta<\e$. 
First note that  
\[
V(w,\Omega)=\displaystyle -\int_{\Omega^c}\frac{\da z}{|z-w|^4}
=\displaystyle -\int_{\Omega^c\cap B_\e(w)}\frac{\da z}{|z-w|^4}
-\int_{\Omega^c\setminus\left(\Omega^c\cap B_\e(w)\right)}\frac{\da z}{|z-w|^4}\,.
\]
The second term of the right hand side can be estimated as $O(1)$ since 
\[
\int_{\Omega^c\setminus\left(\Omega^c\cap B_\e(w)\right)}\frac{\da z}{|z-w|^4}
<\int_{\left(B_\e(w)\right)^c}\frac{\da z}{|z-w|^4}
=\frac{\pi}{\e^2}\,.
\]
It remains to estimate 
\begin{equation}\label{eq_prop_44}
 \int_{\Omega^c\cap B_\e(w)}\frac{\da z}{|z-w|^4}
=-\frac12\int_{\partial(\Omega^c\cap B_\e(w))}\frac{\textrm{\rm det}\big( p-w, \overrightarrow{dp}\,\big)}{|p-w|^4}
\end{equation}
{by a series in $\frac1\delta$, where we have applied} the contour integral expression \eqref{V_contour_int}. 
First note that the boundary of $\Omega^c\cap B_\e(w)$ can be divided into two parts, $\partial\Omega\cap B_\e(w)$ and $\partial B_\e(w)\cap\Omega^c$. %\\ \indent
We may assume, after a motion of $\mathbb R^2$, that the point $w$ is given by $w=(0,\delta)$ and $\partial\Omega\cap B_\e(w)$ is expressed as 
\[
p(s)=\left(s-\frac{\kappa^2}6s^3-\frac{\kappa\kappa'}8s^4+\cdots, \, \frac{\kappa}2s^2+\frac{\kappa'}6s^3+\cdots\right)
\]
by Bouquet's formula, {where $\kappa$ and $\kappa'$ mean $\kappa(0)$ and $\kappa'(0)$ respectively}. 

Let us first estimate the contribution of $\partial\Omega\cap B_\e(w)$ to the contour integral \eqref{eq_prop_44}. 
As 
\[
p'(s)=\left(1-\frac{\kappa^2}2s^2-\frac{\kappa\kappa'}2s^3+\cdots, \, {\kappa}s+\frac{\kappa'}2s^2+\cdots\right)
\]
the numerator of \eqref{eq_prop_44} can be expanded in $s$ as 
\[
\textrm{\rm det}\big( p-w, \overrightarrow{dp}\,\big)%ds
=-\left(\delta+\frac\kappa2(1-\kappa\delta)s^2+\left(\frac{\kappa'}3-\frac{\kappa\kappa'}2\delta\right) s^3+O(1)s^4\right)ds.
\]
By computation we have 
\begin{eqnarray}
|p-w|^2&=&\displaystyle \delta^2+(1-\kappa\delta)s^2-\frac{\kappa'}3\delta s^3+O(1)s^4 \label{eq_p-w_square}\\
&=&\displaystyle \left(\delta^2+(1-\kappa\delta)s^2\right)\left(1+\frac{O(1)\delta s+O(1)s^2}{\delta^2+(1-\kappa\delta)s^2}s^2\right) \nonumber \\
&=&\displaystyle \left(\delta^2+(1-\kappa\delta)s^2\right)\left(1+O(1)s^2\right), \label{eq_p-w_square2}
\end{eqnarray}
which gives an estimate for the denominator of \eqref{eq_prop_44}. 

The range of integration can be estimated as follows. 
Let $s_-<0$ and $s_+>0$ be parameters when $p(s)$ passes through $\partial B_\e(w)$. 
As $|s_\pm|\sim\e$ and $|\delta|\ll\e$, \eqref{eq_p-w_square2} implies that 
\begin{equation}\label{eq_s_pm}
s_\pm=\pm\sqrt{\frac{\e^2-\delta^2}{1-\kappa\delta}}+O(\e^3). 
\end{equation}
Therefore, the contribution of $\partial\Omega\cap B_\e(w)$ to the contour intgeral \eqref{eq_prop_44} can be estimated as 
\begin{eqnarray}
&&\displaystyle -\frac12\int_{\partial\Omega\cap B_\e(w)}\frac{\textrm{\rm det}\big( p-w, \overrightarrow{dp}\,\big)}{|p-w|^4}
=\displaystyle \frac12\int_{s_-}^{s_+}
\frac{\delta+\frac\kappa2(1-\kappa\delta)s^2+O(1) s^3+O(1)s^4}{\left(\left(\delta^2+(1-\kappa\delta)s^2\right)\left(1+O(1)s^2\right)\right)^2}\,ds \nonumber\\[2mm]
&&=\displaystyle \frac1{2\delta^2\sqrt{1-\kappa\delta}}
\int_{t_-}^{t_+}
\frac{1+\frac{\kappa\delta}2t^2+O(\delta^2)t^3+O(\delta^3)t^4}{(t^2+1)^2}\,dt\qquad \left(t=\frac{\sqrt{1-\kappa\delta}}{\delta}s\right), \label{eq_46}
\end{eqnarray}
where $t_\pm$ are given by 
\[
t_\pm=\pm\frac{\sqrt{\e^2-\delta^2}}\delta+\frac{O(\e^3)}\delta.
\]
Recall
\[\begin{array}{rlrl}
\displaystyle \int\frac{1}{(t^2+1)^2}\,dt&=\displaystyle \frac12\left(\arctan t+\frac{t}{t^2+1}\right), \hspace{0.2cm} & \displaystyle \int\frac{t^2}{(t^2+1)^2}\,dt&=\displaystyle \frac12\left(\arctan t-\frac{t}{t^2+1}\right), \\[4mm]
\displaystyle \int\frac{t^3}{(t^2+1)^2}\,dt&=\displaystyle \frac12\left(\frac{1}{t^2+1}+\log(t^2+1)\right), \hspace{0.2cm} & \displaystyle \int\frac{t^4}{(t^2+1)^2}\,dt&=\displaystyle t+\frac{t}{2(t^2+1)}-\frac32\arctan t. 
\end{array}\]
Direct computations imply that the contribution of the $t^3$ and $t^4$ terms in the numerator of \eqref{eq_46} are at most $O(\e)$. 
As $\frac1{t_\pm}=O(\delta)$ the rest can be estimated as 
\begin{eqnarray}
&&\frac1{2\delta^2\sqrt{1-\kappa\delta}}
\int_{t_-}^{t_+}
\frac{1+\frac{\kappa\delta}2t^2}{(t^2+1)^2}\,dt \nonumber\\%[4mm]
&&=\frac1{4\delta^2\sqrt{1-\kappa\delta}}
\left[\left(1+\frac{\kappa\delta}2\right)\arctan t+\left(1-\frac{\kappa\delta}2\right)\frac{t}{t^2+1}\right]_{t_-}^{t_+} \label{eq_key_estimate1}\\%[4mm]
&&=\frac{1+\frac{\kappa\delta}2}{4\delta^2}\left\{
\left(1+\frac{\kappa\delta}2\right)\left\{\left(\frac\pi2-\frac1{t_+}\right)-\left(-\frac\pi2-\frac1{t_-}\right)\right\}+\left(1-\frac{\kappa\delta}2\right)\left(\frac1{t_+}-\frac1{t_-}\right)
\right\}+O(\delta) \nonumber\\%[4mm]
&&=\frac{1+{\kappa\delta}}{4\delta^2}\,\pi+O(\delta). \nonumber
\end{eqnarray}

On the other hand, the contribution of $\partial B_\e(w)\cap\Omega^c$ can be estimated by 
\begin{eqnarray}\label{}
\displaystyle -\frac12\int_{\partial B_\e(w)\cap\Omega^c}\frac{\textrm{\rm det}\big( p-w, \overrightarrow{dp}\,\big)}{|p-w|^4}
&=&-\frac{L(\partial B_\e(w)\cap\Omega^c)}{2\e^3} \nonumber\\%[4mm]
&=&-\frac1{2\e^2}\left(2\arccos\left(\frac\delta\e\right)+\kappa\sqrt{\e^2-\delta^2}\right)+O(\e) \label{eq_key_estimate2}\\%[4mm]
&=&O(1) \nonumber
\end{eqnarray}
as we have fixed $\e$. 
This completes the proof. 
\end{proof}

\subsection{Renormalized energy of planar domains}
The renormalized potential $V(w,\Omega)$ is not integrable over $\Omega$. Hence we define the following renormalization.

\begin{definition} \rm \label{defdomain}
We define the {\em renormalized M\"obius energy} of the domain $\Omega$ by 
\begin{eqnarray}\label{energy}
 E(\Omega)&=&\lim_{\delta\to 0}\left(\int_{\Omega_\delta}V(w,\Omega)\da w+\frac{\pi}{4\delta}L(K)\right),
\end{eqnarray}
where $\Omega_\delta=\{w\in\Omega\,|\,d(w,K)\geq \delta\}$, and $L(K)$ denotes the length of the boundary.
\end{definition}
In some sense, $E(\Omega)$ is also a renormalization of $E(\Omega,\Omega^c)$. Indeed, \eqref{f_V_Omega^c} shows 
\begin{equation}\label{obs}
 E(\Omega)=\lim_{\delta\to 0}\left(\frac{\pi}{4\delta}L(K)-\int_{\Omega_\delta\times\Omega^c}\frac{\da w\da z}{|z-w|^4}\right),
\end{equation}
where  $\Omega_\delta=\{w\in\Omega\,|\,d(w,K)\geq \delta\}$.

Altough $\da w\da z/|z-w|^4$ is invariant under M\"obius transformations, it is not clear at this moment that $E(\Omega)$ is M\"obius invariant. This will be shown later. First we must prove the convergence of \eqref{energy}.  %We will see later that it is positive (cf. Corollary \ref{cor_E'_single}). 

\begin{proposition} \label{welldefined}
Given $\Omega\subset\mathbb R^2$ with compact smooth boundary, the limit in \eqref{energy} exists and is finite.
\end{proposition}
\begin{proof}
{Let $R_0$ be the minimum of the radius of the curvature of $K$.} 
Fix a positive number $\e$ so that $\e\ll R_0$. 
If $0<\delta<\e$ we have 
\[
\int_{\Omega_\delta}V(w,\Omega)\da w=\int_{\Omega_{\e}}V(w,\Omega)\da w
+\int_{N_{\e}(K)\setminus N_\delta(K)}V(w,\Omega)\da w,
\]
where $N_\delta(K)=\{w\in\Omega\,|\,d(w,K)\le\delta\}$. 
If we express the arc-length parameter of $K$ by $s$, 
Proposition \ref{lem_V_Laurent_e} implies 
\[\begin{array}{rcl}
\displaystyle \int_{N_{\e}(K)\setminus N_\delta(K)}V(w,\Omega)\da w
&=&\displaystyle -\frac\pi4\int_0^{L(K)}\int_\delta^{\e}\left(1-\kappa(s)t\right)\left(\frac1{t^2}+\frac{\kappa(s)}t\right)\,dtds+O(\e)\\[4mm]
&=&\displaystyle -\frac\pi4\left(\frac1\delta-\frac1{\e}\right)L(K)-\frac\pi4(\e-\delta)\int_0^{L(K)}\kappa^2(s)ds+O(\e),
\end{array}\]
which completes the proof. 
\end{proof}

\begin{corollary}\label{cor_Omega_e} 
If $\Omega\subset\mathbb R^2$ is a planar domain with smooth regular boundary $K$ of length $L(K)$, then 
\begin{eqnarray*}
E(\Omega)&=&\displaystyle \lim_{\delta\to 0}\left(\frac{\pi}{4\delta}L(K)-E(\Omega_\delta,\Omega^c)\right)\\[4mm]
&=&\displaystyle \lim_{\delta\to 0}\left(\frac{\pi}{4\delta}L(K)-\frac{1}{2}\int_{K_\delta\times K}\cos\theta_p\cos\theta_q\frac{dpdq}{|q-p|^2}\right)\label{f_Omega_e_Re} \\[4mm]
&=&\displaystyle \lim_{\delta\to 0}\left(\frac{\pi}{4\delta}L(K)-\frac{1}{2}\int_{K_\delta\times K}\sin\theta_p\sin\theta_q\frac{dpdq}{|q-p|^2}\right)\\[4mm]
&=&\displaystyle \lim_{\delta\to 0}\left(\frac{\pi}{4\delta}L(K)-\frac{1}{4}\int_{K_\delta\times K}\frac{\overrightarrow{dp}\cdot\overrightarrow{dq}}{|q-p|^2}\right),\notag
\end{eqnarray*}
where $K_\delta=\partial\Omega_\delta$, and $\theta_p$ (resp.$\theta_q$) denotes the angle between the positive tangent vector of $K$ {(resp. $K_\delta$)} at $p$ (resp. at $q$) and $q-p$.
\end{corollary}
\begin{proof}
 The first equality is immediate from \eqref{obs}. The rest follows respectively from \eqref{rere},\eqref{imim} and \eqref{f_E12_contour}. 
{Remark that the orientations of $K$ as $K=\partial\Omega$ and $K=\partial\Omega^c$ are opposite. The signs in the last three lines follow from this fact. }
\end{proof}

\begin{theorem}\label{thm_dxdy_K_Kepsilon_plane}
If $\Omega\subset\mathbb R^2$ is a planar domain with smooth regular boundary $K$ of length $L(K)$, then
\begin{equation}\label{eq_thm48}
E(\Omega)=\lim_{\e\to 0}\left(\frac{L(K)}{2\e}-\frac{1}{4}\int_{K\times K\setminus \Delta_\e}\frac{\overrightarrow{dp}\cdot \overrightarrow{dq}}{|q-p|^2}\right)+\frac{\pi}{8}\int_K \kappa(s)ds,
\end{equation}
where $\kappa$ is the curvature of $K$, with the orientation induced by $\Omega$. 
\end{theorem}
We postpone the proof to Proposition \ref{prop_E_delta_parallel}  where the latter equality is generalized to space curves.

\begin{proposition}
We have 
\begin{equation}\label{eq_prop49}
\lim_{\e\to 0}\left(\frac{L(K)}{\e}-\frac12\int_{K\times K\setminus\Delta_\e}\cos\theta_p\cos\theta_q\frac{dpdq}{|q-p|^2}\right)=-\frac12\int_{K\times K}\sin\theta_p\sin\theta_q\frac{dpdq}{|q-p|^2}.
\end{equation}
\end{proposition}
Again, the proof is postponeed to next section (cf. Propositions \ref{sinsinbp} and \ref{prop_E_space_cos_cos}).

Finally we arrive at an expression of the energy that involves no renormalization.
\begin{theorem}\label{thm_cos_cos_K_Kepsilon_plane}
If $\Omega\subset\mathbb R^2$ is a planar domain with smooth regular boundary $K$ of length $L(K)$, then
\begin{eqnarray}\label{coscos}
E(\Omega)&=&\lim_{\e\to 0}\left(\frac{L(K)}{\e}-\frac{1}{2}\int_{K\times K\setminus \Delta_\e}\cos\theta_p\cos\theta_q\,\frac{dpdq}{|q-p|^2}\right)+\frac{\pi}{8}\int_K \kappa(s)ds\\
&=&-\frac12\int_{K\times K}\sin\theta_p\sin\theta_q\,\frac{dpdq}{|q-p|^2}+\frac{\pi}{8}\int_K \kappa(s)ds\label{sinsin}
,
\end{eqnarray}
where $\kappa$ is the curvature of $K$, with the orientation induced by $\Omega$. 
\end{theorem}
Note the absence of limit in \eqref{sinsin}. We remark that the first term in \eqref{sinsin} is equal (up to {a} factor) to the symmetric energy of \cite{Bu-Si} when $K$ is a single convex curve. %\jun{is ``convex'' necessary?} \gil{I think so: there are absolute values in the definition of the symmetric energy}
In case $\Omega$ is compact, the last term in \eqref{coscos} and \eqref{sinsin} is $\pi^2\chi(\Omega)/4$.
\begin{proof}
As both $\sin\theta_p$ and $\sin\theta_q$ are $O(|q-p|)$ as we will see in \eqref{cota}, 
\[
\lim_{\e\to 0}\int_{K\times K\setminus \Delta_\e}\sin\theta_p\sin\theta_q\,\frac{dpdq}{|q-p|^2}
=\int_{K\times K}\sin\theta_p\sin\theta_q\,\frac{dpdq}{|q-p|^2}\,. 
\]
By applying \eqref{eq_cor33}, \eqref{eq_prop49}, and the above equation to \eqref{eq_thm48}, we obtain the conclusion. %Corollary \ref{cor_33}. 
\end{proof}

\subsection{Renormalized energy of plane curves}\label{subsection_second_renormalization}
Let us introduce an alternative renormalization of $E(\Omega, \Omega^c)$ (cf. \eqref{obs}). 
\begin{definition}\rm\label{defcurve}
 Let $K\subset\mathbb R^2$ be a smooth compact curve bounding a region $\Omega\subset\mathbb R^2$. We define the {\em renormalized energy} of the curve $K$ as  \[
  E(K)=\lim_{\e \to 0}\left(\frac{2L(K)}{\e}-\int_{\Omega\times\Omega^c\setminus\Delta_\e}\frac{\da w\da z}{|z-w|^4}\right),
 \]
where $\Delta_\e=\{(w,z)|\,|z-w|\leq \e\}$, a neighborhood of the diagonal in $\mathbb R^2\times\mathbb R^2$.
\end{definition}
It must be noticed that this energy $E(K)$ does not coincide with the classical M\"obius energy of curves introduced in \cite{jun}.

In the following, $\Omega\subset \mathbb R^2$ will always denote the {\em compact} domain bounded by $K$. 
 The previous energy can be also considered as a renormalization of $E(\Omega,\Omega)$. Indeed, recalling that\[
 \int_{\Omega\times\Omega^c\setminus\Delta_\e}\frac{\da w\da z}{|z-w|^4}+\int_{\Omega\times\Omega\setminus\Delta_\e}\frac{\da w\da z}{|z-w|^4}=\frac{\pi A(\Omega)}{\e^2}\,,
\]
where $A(\Omega)$ denotes the area of the domain, we get 
\[
 E(K)=\lim_{\e\rightarrow 0}\left(\int_{\Omega\times\Omega\setminus\Delta_\e}\frac{\da w\da z}{|z-w|^4}-\frac{\pi}{\e^2}A(\Omega)+\frac{2}{\e}L(K)\right).%+\frac{\pi^2}{4}\chi(\Omega). 
\]

The following notation will be convenient. Let $\Omega$ induce an orientation on  $K=\partial\Omega$. Given $w,z\in \mathbb R^2$ let us consider the linking number 
\[
 \lambda(w,z,K)=\sum_{x\in[wz]\cap K}\epsilon(x)
\]
where $[wz]$ denotes the (oriented) line segment from $z$ to $w$, and $\epsilon$ is the sign of the intersection. Of course, $\lambda(w,z,K)=0$ if $w,z$ are both in $\Omega$ or both in $\Omega^c$. Otherwise $\lambda(w,z,K)=\pm 1$.
\begin{proposition}With the notation introduced above,
 \[
  E(K)=\frac12\int_{\Delta^c}(\#([wz]\cap K)-\lambda^2(w,z,K))\frac{\da w\da z}{|z-w|^4}.
 \]
\end{proposition}
\begin{proof}Let $m=\frac12(z+w)$, $r=|z-w|$, and $\theta$ be the angle between $[wz]$ and any fixed direction. Then 
 \[
  \da w\da z=-r\da md\theta dr. 
 \]
Thus 
\[
 \int_{\Delta_\e^c} \#([wz]\cap K)\frac{\da w\da z}{|z-w|^4}=\int_\e^\infty\left(\int_0^{2\pi} \int_{\mathbb R^2} \# ([wz]\cap K) \da md\theta\right)\frac{dr}{r^3}.
\]
Fixed $r>0$, the  integral between brackets runs over all the positions of an oriented segment of length $r$, and $\da md\theta$ is the Haar measure of the group of rigid plane motions. Hence Poincar\'e's formula gives
\[
\int_\epsilon^\infty\int_0^{2\pi} \int_{\mathbb R^2} \# ([wz]\cap K) \da md\theta\frac{dr}{r^3}=\int_\e^\infty 4r L(K)\frac{dr}{r^3}=\frac{4}{\e}L(K).
\]
Finally 
\[
\lim_{\e\to 0}\int_{\Delta_\e^c}(\#([wz]\cap K)-\lambda^2(w,z,K))\frac{\da w\da z}{|z-w|^4}=\lim_{\e\to 0}\left(\frac{4L(K)}{\e}-2\int_{\Omega\times\Omega^c\setminus\Delta_\e}\frac{\da w\da z}{|z-w|^4}\right).
\]
\end{proof}

The two energies $E(\Omega)$ and $E(K)$ do not coincide, but they are related as follows.

\begin{proposition}\label{relation}
\[E(K)=-\frac12\int_{K\times K}\sin\theta_p\sin\theta_q\frac{dpdq}{|q-p|^2}=E(\Omega)-\frac{\pi^2}{4}\chi(\Omega)\,.\]
\end{proposition} 

\begin{proof}Let us now consider the space $A(1,2)$ of lines in $\mathbb R^2$. This is a 2-dimensional manifold admiting an invariant measure given by $d\ell=dr\wedge d\theta$ where $(r,\theta)$ are the polar coordinates of the point in $\ell$ that is closest to the origin. We can describe each pair $(w,z)\in \mathbb R^2\times\mathbb R^2\setminus\Delta$ by the line $\ell$ through them, and two arc-length parameters $t,s$ along $\ell$. With this notation we have (cf. \cite[equation (4.2)]{santalo})
\[
 \da w\da z=|t-s|dtdsd\ell.
\]
On the other hand,
\[
 \#([wz]\cap K)-\lambda^2(w,z,K)=\sum \epsilon(p)\epsilon(q),\]
where the sum runs over all ordered pairs of distinct points $p,q$ in $[wz]\cap K$.  
 Hence, by the previous proposition, 
\begin{equation}\label{meanchord}
 E(K)=-\frac{1}{2}\int_{A(1,2)}\sum_{p,q\in \ell\cap
K}\frac{\epsilon(p)\epsilon(q)}{|q-p|}d\ell.
\end{equation}
It was shown in \cite{pohl} that for any measurable function $f$ on $K\times K$
\[
 \int_{A(1,2)}\sum_{p,q\in \ell\cap
K}f(p,q){\epsilon(p)\epsilon(q)}d\ell=\int_{K\times K}f(p,q)\sin\theta_p\sin\theta_q \frac{dpdq}{|q-p|}.
\]
Taking $f(p,q)=|q-p|^{-1}$, the result follows. 
\end{proof}

Together with Theorem \ref{thm_dxdy_K_Kepsilon_plane} we have 
\begin{corollary}
When $K$ is a simple closed curve 
\[
E(K)=\lim_{\e\to 0}\left(\frac{L(K)}{2\e}-\frac{1}{4}\int_{K\times K\setminus \Delta_\e}\frac{\overrightarrow{dp}\cdot \overrightarrow{dq}}{|q-p|^2}\right).
\]
\end{corollary}

\begin{remark}\rm\label{prop_E'_int_geom}
In particular, if $\Omega$ is convex, we have from \eqref{meanchord} that
\[
 E(K)=\int_{A(1,2)}\frac{1}{L(\ell\cap \Omega)}d\ell,%+\frac{\pi^2}{4},
\]
where $L(\ell\cap\Omega)$ is the length of the chord. This extends in some sense the Crofton formulas discussed in \cite[chapter 4]{santalo}.
\end{remark}

\bigskip It turns out that $E(K)$ appeared in a Gauss-Bonnet formula for complete surfaces in hyperbolic space, with a tame behaviour at infinity. Before recalling the formula, let us describe this condition on the asymptotic behaviour of the surfaces.

\begin{definition}\label{conelike}\rm
Let $f\colon S\looparrowright \mathbb H^n$ be an immersion of a  $\mathcal C^2$-differentiable surface $S$ in hyperbolic space. We say $S$
has \emph{cone-like ends} if
\begin{enumerate}
\item[i)] $S$ is the interior of a compact surface with boundary  $\overline S$, and taking the
Poincar\'{e} half-space model of hyperbolic space,  $f$ extends to a  $\mathcal C^2$-differentiable immersion $f:\overline{S}\looparrowright\mathbb R^n$,
\item[ii)] $C=f(\partial \overline S)$ is a collection of simple closed curves contained in $\partial_\infty \mathbb H^n$, the boundary of the model, and
\item[iii)] $f(\overline S)$ is orthogonal to $\partial_\infty\mathbb H^n$ along $C$.

\end{enumerate}
\end{definition}

\begin{proposition}[\cite{Gil}]\label{gbtres} Let $S\subset \mathbb H^3\subset\mathbb R^3$ be a surface in Poincar\'e half-space model with cone-like ends on the curve $K=\partial_\infty S\subset \partial_\infty \mathbb H^3\equiv \mathbb R^2$. %\gil{($\partial\mapsto\partial_\infty$)} 
Then the following holds with $\delta(K)$ depending only on $K$
\begin{equation}\label{gbhyp}
\int_S \kappa\  dS=2\pi\chi(S) +\frac{2}{\pi}\int_{\mathbb R^2\times\mathbb R^2}(\#(\ell_{wz}\cap S)-\lambda^2(w,z,K))\frac{\da w\da z}{|z-w|^4}-\delta(K),
\end{equation}
where $\kappa$ denotes the extrinsic curvature of $S$, and $\ell_{wz}$ denotes the geodesic with ideal endpoints $w,z$.
\end{proposition}

Given $K\subset \mathbb R^2$, and $R>0$ we take the surface $S=K\times (0,R\,]\cup \Omega\times \{R\}\subset \mathbb H^3$. By taking limits as $R\to\infty$, the equation above becomes
\begin{equation}\label{eq_prop46}
\delta(K)=\frac{2}{\pi}\int_{\mathbb R^2\times\mathbb R^2}(\#([wz]\cap K)-\lambda^2(w,z,K)))\frac{\da w\da z}{|z-w|^4}=\frac{4}{\pi}E(K).
\end{equation}
\begin{corollary}\label{invariancia}
The energies considered are M\"obius invariant in the following sense: 
\begin{enumerate}
\item If $K\subset\mathbb R^2$ is a smooth closed curve, then $E(K)=E(f(K))$ for every M\"obius transformation $f$ leaving $K$ closed.
\item If $\Omega\subset\mathbb R^2$ is bounded by a closed curve, then $E(\Omega)=E(f(\Omega))$ for every M\"obius transformation $f$ such that $\Omega,f(\Omega)$ are both compact or both unbounded.
\end{enumerate}
\end{corollary}
\begin{proof}
Clearly, $\delta(K)$ is invariant since all other terms in \eqref{gbhyp} are invariant. Together with \eqref{eq_prop46}, this proves the first statement. The second part follows then by Proposition \ref{relation}.
\end{proof}

\begin{proposition}\label{cor_E'_single}Let $\Omega\subset\mathbb R^2$ be a compact domain with smooth boundary $K=\partial\Omega$. Then 
\[\begin{array}{rcl}
E(K)&=&\displaystyle \frac{\pi^2}{2}\chi(\Omega)+\int_{NT(\Omega)}\frac{\da w\da z}{|z-w|^4},
\end{array}\]
where $NT(\Omega)$ is the set of pairs $(w,z)\in \Omega\times\Omega$ such that any circle $\gamma$ containing $w$ and $z$ intersects $K$. 

\end{proposition}
\begin{proof}
Let $Q\subset\mathbb H^3$ be the intersection of all half-spaces containing $\Omega^c$ in its ideal set. This is a kind of convex hull of $\Omega^c$, and is bounded by a surface $S$ of class $C^1$. With the arguments of \cite[Proposition 3]{Gil}, one can approximate $S$ by a sequance of surfaces $S_n$ with cone-like ends and with total curvatures converging to 0. Then Proposition \ref{gbtres} {and \eqref{eq_prop46} give}
\[
 0=2\pi\chi(\Omega)+\frac{2}{\pi}\int_{\mathbb R^2\times\mathbb R^2}(\#(\ell_{wz}\cap S)-\lambda^2(w,z,K)))\frac{\da w\da z}{|z-w|^4}-\frac{4}{\pi}E(K).
\]
Note that the integrand above is 2 if $w,z\in \Omega$ and $\ell_{wz}\cap Q\neq\emptyset$; otherwise it is 0. But $\ell$ meets the convex hull $Q$ if and only if every geodesic plane $\wp$ containing $\ell$ meets $\Omega^c$. 
\end{proof}

\begin{corollary}\label{bounds} Let $\Omega$ be compact with $n$ connected components and let the  boundary $K=\partial \Omega$ have $k$ components. Then  $E(\Omega)\geq (3n+k)\pi^2/4$ with equality only if $n=k=1$ and $\Omega$ is a disk.
 \end{corollary}

\begin{proof}\rm\label{disconnected}
Given a compact domain $\Omega\subset\mathbb R^2$, we have from Theorem \ref{thm_cos_cos_K_Kepsilon_plane} that
\[
 E(\Omega)=E(\mathbb R^2\setminus\Omega)+\frac{\pi^2}{2}\chi(\Omega).
\]

Let $\Omega$ be a compact connected domain with non-connected boundary. Then $\mathbb R^2\setminus\Omega=\cup_{i=1}^{k} \Omega_i$ for a collection of domains $\Omega_i$ (one of them, say $\Omega_1$, non-compact) with connected boundaries $K_i=\partial \Omega_i$. Hence
\begin{equation}\label{expansion}
 E(\Omega)=\sum_{i=1}^{k} E(\Omega_i)+\sum_{i\ne j} E(\Omega_i,\Omega_j)+\frac{\pi^2}{2}\chi(\Omega).
\end{equation}
Clearly $E(\Omega_i,\Omega_j)> 0$, and by the previous proposition
\[
 E(\Omega_i)=E(K_i)+\frac{\pi^2}{4}\geq \frac{3\pi^2}{4}, \mbox{ for } i>1, \qquad E(\Omega_1)=E(K_1)-\frac{\pi^2}{4}\geq \frac{\pi^2}{2}
\]
Hence,
\[
 E(\Omega)\geq \frac{3(k-1)\pi^2}{4}+\frac{\pi^2}{2}=\frac{3(k+1)\pi^2}{4}.
\]
If $\Omega$ has $n$ connected components, we just need to use again that mutual energies are positive.

Suppose now that we have the equality in the inequalities above. Then clearly $k=1$, and $NT(\Omega)$ has empty interior. Let now $D$ be a maximal closed disc contained in $\Omega$. If $\Omega\neq D$, then $\Omega$ has a bigger diameter than $D$. But then $(w,z)\in NT(\Omega)$ whenever $|z-w|$ is close to the diameter of $\Omega$. We conclude that $\Omega=D$.
\end{proof}

%%%%%%%%%%%%%%%%%%%%%%%%%%%%%%%%%%%%%%%%%%%%%%%%%%%%%%%%%%%%%%%%%%%%%%%%%%%%%%%%%%%%%%%%%%%

\section{A new M\"obius invariant functional for space curves}\label{section_5}

% %
\subsection{Mutual energies for space curves}
Let $K_1,K_2\subset\mathbb R^3$ be a pair of disjoint {oriented} space curves. Each of them is the boundary of an orientable surface $\Omega_i$ (Seifert surface), but we will need these surfaces to be disjoint. This is only true if $K_1, K_2$ are splittable. Hence we consider $K_1,K_2\subset\mathbb R^3\subset \mathbb R^n$ for $n\geq 5$. Then there exist disjoint orientable surfaces $\Omega_1,\Omega_2\subset \mathbb R^n$ with $\partial \Omega_i=K_i$. Now we can generalize the definition \ref{mutual} to space curves.

\begin{definition}\label{def_E(K1,K2)}\rm
In the situation described above, we define the mutual energy of $K_1,K_2$ by
\[
E(K_1,K_2)=\frac12\int_{\Omega_1\times\Omega_2}\frac{\omegasub{T^{\ast}\mathbb R^n}}{2}\wedge \frac{\omegasub{T^{\ast}\mathbb R^n}}{2}
\]
where $\omegasub{T^{\ast}\mathbb R^n}$ is the canonical symplectic form of $\mathbb R^n\times\mathbb R^n\setminus\Delta\cong T^{\ast}\mathbb R^n$
\end{definition}

This definition does not depend on the choice of $\Omega_1,\Omega_2$ as shown by the following proposition

\begin{proposition}\label{prop_f_cos_cos<->omega}
In the situation above,  
\begin{equation}\label{f_cos_cos<->omega}
E(K_1,K_2)=-\frac12\int_{K_1\times K_2} \cos\theta_1\cos\theta_2\frac{dp_1dp_2}{|p_2-p_1|^2},
\end{equation}
where $\theta_i\in[0,\pi]$ is the angle between $\overrightarrow{dp_i}$ and $p_2-p_1$. 
\end{proposition}
\begin{proof}
The proof of \eqref{rere} also works here but using
\[
 \lambda=\frac{\sum (z_i-w_i)dz_i}{|z-w|^2},\qquad \rho=\frac{\sum (z_i-w_i)dw_i}{|z-w|^2},\qquad\omega=-\frac12\omegasub{T^{\ast}\mathbb R^n}
\]
so that $d\lambda=d\rho=\omega$ (cf. \eqref{eq_omega_cotan_one-form} and \eqref{eq_folklore}).
\end{proof}

\subsection{Linking with circles}

Next we give an interpretation of $E(K_1,K_2)$ as the average of some linking numbers with circles. Recall the following result of \cite{banchoff.pohl}: given two disjoint oriented curves $K_1,K_2\subset \mathbb R^3$, one has
\[
 \int_{A(1,3)}\lambda(\ell,K_1)\cdot\lambda(\ell,K_2)d\ell=\int_{K_1\times K_2}\cos\theta_1\cos\theta_2 dp_1dp_2
\]
where $A(1,3)$ is the space of  lines $\ell\subset\mathbb R^3$, endowed with an invariant measure $d\ell$, and $\lambda$ denotes the linking number. Note that the integrand on the left hand side is independent of the orientation of $\ell$. It changes sign when we change the orientation of $K_1$ or $K_2$.

We {now look} for an analogue of the previous result in the realm of M\"obius geometry. The role of lines will be played by circles. Let us denote the set of all oriented circles $\gamma\subset\mathbb R^3$ by $\mathcal S(1,3)$. This is a homogeneous space of the M\"obius group $\mbox{M\"ob}_3$ with isotropy group $\mathbb S^1\times \mbox{M\"ob}_1$. Since these are unimodular groups, the space of circles $\mathcal S(1,3)$ admits a measure $d\gamma$ invariant under $\mbox{M\"ob}_3$. Let us describe this measure explicitely. Each circle $\gamma\subset\mathbb R^3$ is uniquely determined by its center $c\in\mathbb R^3$, the radius $r>0$, and a unit vector $n\in\mathbb S^2$ orthogonal to the plane containing $\gamma$. Then the (unique up to a constant factor) M\"obius invariant measure on the space of circles is
\[
 d\gamma=\frac{1}{r^4}drdcdu
\]
where $dc$ is the volume element of $c\in \mathbb R^3$, and $du$ denotes the area element of $u\in\mathbb S^2$. Indeed, the latter is clearly invariant under the group $\mbox{Sim}_3$ generated by rigid motions and homotheties of $\mathbb R^3$. Such transformations act transitively on the space of circles. Hence, every two measures on $\mathcal S(1,3)$ that are invariant under $\mbox{Sim}_3$ must be a constant multiple of each other. But clearly the measures invariant under $\mbox{M\"ob}_3$, which we know exist, are also invariant under $\mbox{Sim}_3$. 

Our next goal is to compute
\[
 I_3(K_1,K_2)=\int_{\mathcal S(1,3)}\lambda(\gamma,K_1)\cdot\lambda(\gamma,K_2)d\gamma.
\]

It will be useful to take $S_1,S_2$ disjoint surfaces with $\partial S_i=K_i$. This is not possible if the curves are linked. 
To solve this we consider {again} $K_1,K_2 \subset \mathbb R^n$ {with $n\geq 5$,} and we consider the general problem of determining
\[
 I_n(K_1,K_2)=\int_{\mathcal S(n-2,n)}\lambda(\xi,K_1)\cdot\lambda(\xi,K_2)d\xi
\]
where $d\xi$ is the conformally invariant measure in the space of oriented codimension 2 spheres $\mathcal S(n-2,n)$. {Just like in the case when} $n=3$, this space admits a M\"obius invariant measure given in terms of the radius $r$, the center $c\in\mathbb R^n$ and a normal direction $u\in\mathbb S^{n-1}$ by
\[
 d\xi=\frac{1}{r^{n+1}}drdcdu.
\]
Note that considering $K_1,K_2\subset \mathbb R^n\subset \mathbb  R^{n+p}$ one has $I_n(K_1,K_2)=c_{n,p} I_{n+p}(K_1,K_2)$ for a constant $c_{n,p}$ to be computed. Therefore, it is enough to consider the problem for $n\geq 5$.
 
Let $\mathcal S(0,n)=\mathbb R^n\times\mathbb R^n\setminus\Delta$ denote the space of point pairs ({oriented} $0$-spheres).  We consider the flag space
\[
 \mathcal F=\{(w,z;\xi)\in \mathcal S(0,n)\times \mathcal S(n-2,n)|w,z\in \xi\}.
\]
There is a natural double fibration
\begin{displaymath}
\xymatrix{
&  \mathcal F \ar[dl]_{\pi_1} \ar[dr]^{\pi_2} &  \\
\mathcal S(0,n) & & \mathcal S(n-2,n)}
\end{displaymath}
with $\pi_1,\pi_2$ the obvious maps. Note that $\mathcal F$ can be identified {with} $S(0,n)\times G^+(2,n)$ where $G^+(2,n)$ denotes the Grassmannian of oriented planes in $\mathbb R^n$. This way $\pi_1$ is just the projection on the first factor.
Note that the dimensions of $\mathcal F, \mathcal S(0,n)$, and $\mathcal S(n-2,n)$ are given by $4n-4, 2n$, and $2n$ respectively.

\begin{proposition}Let $S_1,S_2\subset \mathbb R^n$ be disjoint surfaces with boundary $K_1,K_2$ respectively. Then
 \[
  I_n(K_1,K_2)=\int_{S_1\times S_2}((\pi_1)_*\circ\pi_2^*)(d\xi)=-\frac{\mathrm{vol}(\mathbb S^{n-1})\mathrm{vol}(\mathbb S^{n-2})}{2 n(n-1)\pi}\cdot\int_{S_1\times S_2} \omega_{T^*\mathbb R^n}\wedge\omega_{T^*\mathbb R^n}
\]
\end{proposition}where $(\pi_1)_*$ denotes integration along the fibers of $\pi_1$.

\begin{proof}We show first the second equality. 
Given an orthonormal frame $o;e_1,\dots,e_{n+1}$ of $\mathbb H^{n+1}$, we consider the geodesic $\ell(t)=\exp_o (t e_1)$, and the codimension 2 geodesic plane $L=\exp_o(e_n\wedge e_{n+1})^\bot$. This defines an element $(w,z;\xi)\in\mathcal F$, where $z=\lim_{t\rightarrow -\infty} \ell(t), w=\lim_{t\rightarrow +\infty} \ell(t)$, and $\xi\subset \mathbb S^n$ is  the set of ideal points of $L$. Using such local frames one can write
 \begin{equation}\label{mesures}
  d\xi=\omega_{n+1}\wedge\omega_{1,n+1}\wedge\dots\wedge\omega_{n-1,n+1}\wedge \omega_n\wedge\omega_{1,n}\wedge\dots\omega_{n-1,n},
 \end{equation}
where $\omega_i=\langle do,e_i\rangle$, and $\omega_{ij}=\langle \nabla e_i,e_j\rangle$, {where $\nabla$ is} the riemannian connection of $\mathbb H^{n+1}$. Indeed, the right hand side is a common expression of the isometry invariant measure of (codimension 2) geodesic planes of $\mathbb H^{n+1}$ (cf. \cite{santalo}). Hence both sides coincide except for a constant factor. To find this factor, we assume by invariance that $o=(0,\cdots, 0,1)$ in the half-space model, and $e_i$ is the canonical basis. Then $r=1$, and
$dr=\omega_{n+1}$, $dc=\omega_{1,n+1}\w\cdots \w\omega_{n-1,n+1}\w\omega_n$, {$du=\omega_{1,n}\w\cdots\w\omega_{n-1,n}$}. This shows \eqref{mesures}.

Now, given $(w,z)\in \mathcal S(0,n)$ take a frame $p,u_1,u_2,\dots,u_{n+1}$ 
as above. Then, for any point $(w,z;\xi)\in \pi_1^{-1}(w,z)$ in the fiber we can choose a frame $p;u_1,e_2,\dots,e_{n+1}$. 
Note that,
\[
\begin{array}{c}
\displaystyle 
 \omega_n\wedge\omega_{n+1}=\langle e_n\wedge e_{n+1},dp\wedge dp\rangle=\langle \sum_{2\leq i<j\leq n+1}p_{ij}u_i\wedge u_j,dp\wedge dp\rangle,
\\[6mm]
\displaystyle 
 \omega_{1,n}\wedge\omega_{1,n+1}=\langle e_n\wedge e_{n+1},\nabla e_1\wedge \nabla e_1\rangle=\langle \sum_{2\leq i<j\leq n+1}p_{ij}u_i\wedge u_j,\nabla e_1\wedge \nabla e_1\rangle, 
\end{array}
\]
where $p_{ij}$ are the Pl\"ucker coordinates of $e_n\wedge e_{n+1}$ in $\bigwedge^2 (e_1)^\bot\subset\bigwedge^2 T_p\mathbb H^{n+1}$.
i.e.,
\[
 p_{ij}=\left|\begin{array}{cc}\langle e_n,u_i\rangle&\langle e_n,u_j\rangle\\\langle e_{n+1},u_i\rangle&\langle e_{n+1},u_j\rangle
\end{array}
\right|.
\]
This way, the fiber $\pi_1^{-1}(w,z)$ is identified to a submanifold $P$ (given by the Pl\"ucker relations) of the unit sphere $\mathbb S^{N-1}$ with $N= {n\choose 2}$. 
Thus
\begin{equation}\label{dxi}
 d\xi=\sum_{i<j,r<s}p_{ij}p_{rs}(\langle u_i,dp\rangle\wedge\langle u_j,dp\rangle\wedge\langle u_i,\nabla e_1\rangle\wedge\langle u_j,\nabla e_1\rangle)dP,
\end{equation}where $dP$ is the volume element on $P$ induced by the metric of $\mathbb S^{N-1}$.
Now, since $P\subset \mathbb S^{N-1}$ and $P$ is isometric to the grassmannian of oriented $2$-planes in $\mathbb R^n$, 
\[
%c:=
\int_{P} p_{ij}^2 dP={n \choose 2}^{-1}\int_P\sum_{2\leq r<s\leq n+1}  p_{rs}^2dP={n \choose 2}^{-1}\mathrm{vol}(P)={n \choose 2}^{-1}\frac{\mathrm{vol}(\mathbb S^{n-1})\mathrm{vol}(\mathbb S^{n-2})}{2\pi}.\]
 Let this constant be denoted by $c$. On the other hand, the function $p_{ij}$ is odd with respect to the symmetry of $\mathbb S^{N-1}$ fixing $(u_i)^\bot$. Hence, 
\[\int_P p_{ij} p_{rs}dP=0\quad\mbox{for }\{i,j\}\neq\{r,s\}. 
\]
Therefore
\[
 \pi_{1*}\pi_2^*d\xi=\int_{\pi_1^{-1}(w,z)}\pi_2^*d\xi=c\sum_{2\leq i<j\leq n+1}\langle u_i,dp\rangle\wedge\langle u_j,dp\rangle\wedge\langle u_i,\nabla e_1\rangle\wedge\langle u_j,\nabla e_1\rangle
\]
\[=c\sum_{2\leq i<j\leq n+1}\omega_i\wedge\omega_j\wedge\omega_{1i}\wedge\omega_{1j}
=-\frac{c}{2}d\omega_1\wedge d\omega_1=-\frac{c}{2}\cdot\omega_{T^*\mathbb R^n}\wedge\omega_{T^*\mathbb R^n}.
\]%\gil{signs corrected}
since  $d\omega_1=-\omega_{T^*\mathbb R^n}$.

In order to show the first equality, let us consider the region $U=\pi_1^{-1}(S_1\times S_2)\subset \mathcal F$, and the mapping $\phi=\pi_2|_U\colon U\rightarrow \mathcal S(n-2,n)$. By \eqref{dxi}, one can chek that the multiplicity of $\xi\in\mathcal S(n-2,n)$ as an image value of $\phi$ (taking orientations into account) is given by
\[
\nu(\xi)=\sum_{z\in\xi\cap S_1\\w\in \xi\cap S_2}\epsilon_i(z)\epsilon(w)=(\xi\cdot S_1)(\xi\cdot S_2)=\lambda(\xi,K_1)\lambda(\xi,K_2). 
\]
Here $S(n-2,n)$ was oriented by $d\xi$, and we used the orientation in $U\equiv S_1\times S_2\times G^+(2,n)$ given by $dS^1\wedge dS^2\wedge dP$.
Finally, the coarea formula and integration along the fibers yield
\[
I_n(K_1,K_2)=\int_{\mathcal S(n-2,2)}\nu(\xi)d\xi=\int_{U}\pi_2^*(d\xi)= \int_{\pi_1^{-1}(S_1\times S_2)}\pi_2^*(d\xi)=\int_{S_1\times S_2}(\pi_1)_*\pi_2^*d\xi.
\]
\end{proof}

In particular, 
\[
 c_{n,p}=\frac{\mathrm{vol}(\mathbb S^{n+p})\mathrm{vol}(\mathbb S^{n+p-1})n(n-1)}{\mathrm{vol}(\mathbb S^{n})\mathrm{vol}(\mathbb S^{n-1}){(n+p)(n+p-1)}}.
\]

\begin{corollary}\label{corospace}The mutual energy of a pair of disjoint space curves is given by
 \[
  E(K_1,K_2)=-\frac{3}{16\pi}\int_{\mathcal S(1,3)}\lambda(\gamma,K_1)\lambda(\gamma,K_2)d\gamma.
 \]%\gil{$\Gamma\mapsto \mathcal S(1,3)$, sign corrected.}
\end{corollary}

\begin{remark}\rm Note that $E(K_1,K_2)=0$ does not imply  that every circle is trivially linked with $K_1$ or with $K_2$. An example where the mutual energy vanishes is given by a pair of circles $K_1, K_2$ such that $K_1$ is orthogonal to every sphere containing $K_2$. For instance, $K_1$ can be the line of points $p\in\mathbb R^3$ such that $|q-p|$ is constant for all  $q \in K_2$. 
\end{remark}

\subsection{Renormalized measure of circles linking a space curve}
Let us  now consider a single space curve $K\subset \mathbb R^3$ which is assumed to be closed but not necessarily connected. We will define a functional $E(K)$ such that $E(K_1\cup K_2)=E(K_1)+E(K_2)+2E(K_1,K_2)$ whenever $K_1,K_2$ are disjoint.
Our results are closely analogous to the following formula due to Banchoff and Pohl (cf.\cite{banchoff.pohl})
\begin{equation}\label{banchoffpohl}
 \int_{A(1,3)}\lambda(\ell,K)^2d\ell=-\int_{K\times
K
}{\cos\tau\sin\theta_p\sin\theta_q} {d pd q} =\int_{K\times
K
}{\cos\theta_p\cos\theta_q} {d pd q} , 
\end{equation}
where $\theta_p\in[0,\pi]$ (resp. $\theta_q\in[0,\pi]$) is the angle between $\overrightarrow{dp}$ (resp. $\overrightarrow{dq}$) and $q-p$, and $\tau$ is the angle  between the two oriented planes through $p,q$ tangent to $K$ at $p$ and $q$ respectively. These planes are oriented by  $\overrightarrow{dp}\wedge (q-p)$ and $\overrightarrow{dq}\wedge (q-p)$ respectively.   In order to define $E(K)$ it would be natural to consider
\[
 \int_{S(1,3)}\lambda(\gamma,K)^2d\gamma.
\]
However this integral diverges due to the blow up of the density $d\gamma=r^{-4}drdcdn$ when the radius $r$ goes to $0$. Hence we take the following renormalization.

\begin{definition}\label{def_E_space_curve}
Let $K\subset \mathbb R^3$ be a closed (maybe non-connected) space curve of class $C^2$. We define
 \[
  E(K)=\lim_{\e\to 0}\left(\frac{3\pi L(K)}{8\e}-\frac{3}{16\pi}\int_{\mathcal S_\e(1,3)}\lambda(\gamma,K)^2d\gamma\right),
 \]
where $\mathcal S_\e(1,3)$ is the open set of $\mathcal S(1,3)$ containing the circles of radius $r>\e$.
\end{definition}
The following proposition gives an  expression of $E(K)$ which involves no renormalization. 

\begin{proposition}\label{sinsinbp}The previous limit exists, and coincides with the following integral 
 \[
E(K)=\frac{3}{16\pi}\int_{\mathcal S(1,3)}(\#(K\cap [\gamma])-\lambda(\gamma,K)^2)d\gamma,
 \]
where $[\gamma]$ denotes the disk with boundary $\gamma$. The previous integral converges and coincides with
\begin{equation}\label{sinsinespai}
E(K)=-\frac12\int_{K\times K}{\cos\tau\sin\theta_p\sin\theta_q}\frac{d pd q}{|q-p|^2},
\end{equation}
where $\theta_p,\theta_q$ and $\tau$ are as in \eqref{banchoffpohl}. 
\end{proposition}

\begin{proof}
Let  $A(2,3)$ be the space of oriented affine planes of $\mathbb R^3$, and let $d\wp$ be an invariant measure on it. Then
\begin{equation}\label{measuredisks}\int_{\mathcal S_\e(1,3)}\#(K\cap[\gamma])d\gamma=\int_{A(2,3)}\int_\e^\infty\int_\wp\#(B_c(r)\cap \wp\cap K)\frac{dcdrd\wp}{r^4}=\frac{\pi}{\e}\int_{A(2,3)}\#(\wp\cap K)d\wp=\frac{2\pi^2}{\e}L(K).
\end{equation}
This shows the first equation, except for the convergence of the integral. To see the second part we start with
\[
\int_{S(1,3)}(\#([\gamma]\cap K)-\lambda^2(\gamma,K))d\gamma=-\int_{A(2,3)}\int_0^\infty\frac{1}{4}\int_\wp\sum_{x,y}\epsilon(x)\epsilon(y)\frac{1}{r^4}dcdrd\wp
\]
where the sum runs over the pairs $x,y\in K\cap B_c(r)\cap  \wp$, and $\epsilon(x),\epsilon(y)$ are the intersection signs of $K$ and $\wp$. Now, an elementary computation shows
\[
E(K)=-\frac{1}{\pi}\int_{A(2,3)}\sum_{x,y\in \wp\cap K}\frac{\epsilon(x)\epsilon(y)}{|y-x|}d\wp.
\]
Finally, by the results of Pohl \cite{pohl} we get
\[
 E(K)=-\frac12\int_{K\times K}\cos\tau\sin\theta_p\sin\theta_q\frac{dpdq}{|q-p|^2}. 
\]
We can now check the convergence of the integrals since 
\begin{equation}\label{cota}
 \sin\theta_p=\frac{\kappa(p)}{2}|q-p|+O(|q-p|^2)
\end{equation}
where $\kappa$ denotes the curvature of $K$. Indeed, let $f:(0,\e)\rightarrow \mathbb R^3$ be an embedding with $f((0,\e))\subset K$ and $|f'(s)|=1$ for all $s\in(0,\e)$. Then, for $p=f(s),q=f(t)$
\[
 |\sin\theta_p|=\left|f'(s)\times \frac{(f(t)-f(s))}{|f(t)-f(s)|}\right|=\frac{|f'(s)\times(f'(s)(t-s)+\frac12 f''(s)(t-s)^2+O(|t-s|^3)|}{|f(t)-f(s)|}
\]
\[
 =\frac12|f''(s)||t-s|+O((t-s)^2).
\]
\end{proof}

By Proposition \ref{relation} we have 
\begin{corollary}\label{cor_def_E_2=3}
When $K$ is a planar curve, Definitions \ref{defcurve} and \ref{def_E_space_curve} of the energy $E(K)$ coincide. 
\end{corollary}

 In particular   $E(K)> 0$ for $K$ planar and convex. This explains the choice of the sign in the definition of $E$. However, for space curves there is no lower (nor upper) bound of $E$. Indeed, if two arcs of $K$ come close to each other (not othogonally) then $E(K)$ blows up to $\pm\infty$.

\begin{remark}\rm
The functional $E$ is continuous with respect to the topology of uniform $C^2$ convergence. Even more, suppose a sequence of closed curves  $K_n\subset \mathbb R^3$ converging pointwise to a closed embedded curve $K\subset\mathbb R^3$ in the $C^1$ topology, and with uniformly bounded curvature. Then $\lim_{n\to\infty}E(K_n)=E(K)$. This follows from \eqref{sinsinespai}, and Lebesgue's dominated convergence theorem, which applies here in virtue of \eqref{cota}.
\end{remark}

It is interesting to recall that the {\em writhe} of $K$ is given by
\[
 W(K)=\frac{1}{4\pi}\int_{K\times K}{\sin\tau\sin\theta_p\sin\theta_q}\frac{d pd q}{|q-p|^2}.
\]
Recall also that $W(K)$ is the average of signed self-intersections of projections of $K$. 
A remarkable fact is that the writhe is invariant under orientation preserving M\"obius transformations {(cf. \cite{{BW}})}. 
{We will see below that $E(K)$ not only has an integral expression similar to $W(K)$, but it shares also this invariance}.

\begin{proposition}\label{prop_E_space_cos_cos}
\begin{equation}\label{coscosespai}
E(K)= \lim_{\e\to 0}\left(\frac{L(K)}{\e}-\frac12\int_{K\times K\setminus\Delta_\e}\cos\theta_p\cos\theta_q\frac{dpdq}{|q-p|^2}\right).
\end{equation}
\end{proposition}

\begin{proof}We use intergation by parts, as in Proposition 5 of \cite{banchoff.pohl}. Given $p,q\in K\times K\setminus \Delta$ let $e_1,e_2,e_3$ be an orthonormal moving frame (locally defined on $K\times K\setminus \Delta$ with $e_1=(q-p)/|q-p|$, and $e_3\bot T_pK$. As usual let $\omega_i=dp\cdot e_i$, and $\omega_{ij}=de_i\cdot e_j$. Then 
\[
\cos\theta_p\cos{\theta_q}\frac{d p\wedge d q}{|q-p|^2}=-\frac{d (|q-p|)\wedge\omega_{1}}{|q-p|^2}=
d \left(\frac{1}{|q-p|}\right)\wedge\omega_{1}
\]
\[
=d \left(\frac{\omega_{1}}{|q-p|}\right)-\frac{1}{|q-p|}d \omega_{1}=
d \left(\frac{\omega_{1}}{|q-p|}\right)-\frac{\omega_{12}\wedge\omega_{2}}{|q-p|}-\frac{\omega_{13}\wedge\omega_{3}}{|q-p|}
\]
\[
=d \left(\frac{\omega_{1}}{|q-p|}\right)+\cos\tau\sin\theta_p\sin\theta_q\frac{d p\wedge d q}{|q-p|^2}.
\]
On the other hand 
\[
\int_{K\times
K\setminus\Delta_\e}d \frac{\omega_{12}}{|q-p|}=\int_{\partial
\Delta_\e }\frac{\omega_{12}}{|q-p|}=2\int_K
\frac{1}{\e}d q +O(\e)=\frac{2}{\e}L(K)+O(\e).\]
\end{proof}

Propositions \ref{prop_f_cos_cos<->omega} and \ref{prop_E_space_cos_cos} imply 
\begin{corollary}\label{cor_E_K1_cup_K2}
Let $K_1,K_2$ be a pair of disjoint oriented curves. 
Then 
\begin{equation}\label{eq_E_K1_cup_K2}
 E(K_1\cup K_2)=E(K_1)+E(K_2)+2E(K_1,K_2).
\end{equation}
\end{corollary}
With the equation above and Definition \ref{def_E_space_curve} we recover Corollary \ref{corospace}.

\begin{proposition}\label{renocirc}
We have 
\[
E(K)=\lim_{\e\to0}\left(\frac{L(K)}{2\e}-\frac14\int_{K\times K\setminus\Delta_\e}\frac{\overrightarrow{dp}\cdot\overrightarrow{dq}}{|q-p|^2}\right), 
\]
where $\overrightarrow{dp}\cdot\overrightarrow{dq}=dp_1\wedge dq_1+dp_2\wedge dq_2+dp_3\wedge dq_3$.
\end{proposition}
\begin{proof}
 It is elementary to see
\[
 \overrightarrow{dp}\cdot\overrightarrow{dq}=(\cos\theta_p\cos\theta_q+\cos\tau\sin\theta_p\sin\theta_q)dpdq.
\]
Now averaging \eqref{sinsinespai} and \eqref{coscosespai} gives the result.
\end{proof}

\begin{corollary}\label{E_space_K_1_K2}
Let $K_1,K_2$ be a pair of disjoint curves. 
Then 
\[
E(K_1,K_2)=-\frac14\int_{K_1\times K_2}\frac{\overrightarrow{dp}\cdot\overrightarrow{dq}}{|q-p|^2}.
\]
\end{corollary}

\subsection{Gauss-Bonnet theorem for complete surfaces in hyperbolic space}
Next we show the M\"obius invariance of $E(K)$.
To this end we will use a Gauss-Bonnet formula for complete surfaces in hyperbolic space.

Let $S\subset \mathbb H^4$ be a surface in hyperbolic 4-space (Poincar\'e model) with cone-like ends on the curve $K\subset\mathbb R^3=\partial_\infty\mathbb H^4$. Given an element $(x,n)\in N^1S$, the unit normal bundle of $S$, the Lipschitz-Killing curvature $\kappa(x,n)$ is defined as the determinant of the endomorphism $dn_{(x,n)}$ of $T_{(x,n)}(N^1S)$. We are interested in the integral of $\kappa(x,n)$ along the fibers $N_x^1S$ of $N^1S$. Using Gauss equation one gets easily
\begin{equation}\label{gauss}
 \frac{1}{\pi}\int_{N_x^1S}\kappa(x,e)de=\kappa_i(x)+1
\end{equation}
where $de$ is the volume element on  $N_x^1S$, and $\kappa_i$ denotes the Gauss (intrinsic) curvature of $S$. The additive constant $1$ comes from the  sectional curvature of the ambient space $\mathbb H^4$. Given $\e>0$ let $S_\e:=\{x\in S|x^4\geq \e\}$. Then, the classical intrinsic Gauss-Bonnet formula gives
\[
 \int_{S_\e}(\kappa_i+1)dS=2\pi\chi(S_\e)+A(S_\e)-\int_{\partial S_\e}k_g=2\pi\chi(S_\e)+A(S_\e)-\frac{L(K)}{\e}+O(\e)
\]
where  $k_g$ is the geodesic curvature in $S_\e$. We used $k_g=1+O(\e^2)$, and the fact that euclidean lengths of $\partial S_e$ and $K$ have a difference of order $\e^2$. Taking limit $\e\to 0$ we get
\begin{equation}\label{gb1}
\frac1\pi\int_{N^1S}\kappa(x,e)dedS= \int_S (\kappa_i(x)+1)dS=2\pi\chi(S)+\lim_{\e\to 0}\left(A(S_\e)-\frac{L(K)}{\e}\right).
\end{equation}The convergence of the integrals follows from the hypothesis that $S$ has cone-like ends by the same arguments as Proposition 7 in \cite{Gil}.
This formula appeared in a more general setting in \cite{alexakis}. The latter limit was called {\em renormalized area} of $S$. Here we will use a different renormalization that leads to the same value.

\begin{proposition}Let $\mathcal L_2^+$ denote the space of oriented $2$-dimensional geodesic planes in $\mathbb H^4$. Let $\mathcal L_{2,\e}^+$ be the subset of $\mathcal L_2^+$ containing the planes which define a circle in $\mathbb R^3=\partial_\infty \mathbb H^4$ of radius bigger than $\e$. Then the renormalized area of $S$ is given by
 \[\begin{array}{rcl}
\displaystyle \lim_{\e\to 0}\left(A(S_\e)-\frac{L(K)}{\e}\right)&=&\displaystyle \lim_{\e\to 0}\left(\frac{3}{4\pi^2} \int_{\mathcal L_2^+}\#(\ell\cap S_\e)d\ell-\frac{L(K)}{\e}\right)\\[3mm]
&=&\displaystyle \lim_{\e\to 0}\left(\frac{3}{4\pi^2}\int_{\mathcal L_{2,\e}^+}\#(\ell\cap S)d\ell-\frac{3{L(K)}}{2{\e}}\right),
\end{array} \]
{where $d\ell$ is the pull-back of $d\gamma$ through $\mathcal L_2^+\ni\ell\mapsto \gamma=\partial_\infty \ell\in\mathcal S(1,3)$. }
\end{proposition}
\begin{proof}
 The first equality follows immediately from Crofton formula (cf. \cite[p.245]{santalo}). In order to check the second equality, we need the following claim: given two surfaces $R,S\subset \mathbb H^4$ with cone-like ends on the same ideal curve $K\subset \partial_\infty\mathbb H^4$, one has
\begin{equation}\label{wondering}
 \lim_{\e\to 0}\int_{\mathcal L_2^+}(\#(\ell\cap R_\e)-\#(\ell\cap S_\e))d\ell=\lim_{\e\to 0}\int_{\mathcal L_{2,\e}^+}(\#(\ell\cap R)-\#(\ell\cap S))d\ell.
\end{equation}
Indeed, given $\e\geq 0$ we consider 
\[
 E_\e=\{(p,\ell)\in \mathbb H^4\times\mathcal L_2\ |\ p\in\ell\cap(S_\e\cup R_\e)\}
\]
endowed with the measure given by the pull-back of $d\ell$ through the projection $E_\e\rightarrow \mathcal L_2$. We consider also the function $f$ defined on $E_0$ by pulling back the difference of indicator functions  $\mathbf{1}_S-\mathbf{1}_R$. Then \eqref{wondering} is equivalent to
\[
 \lim_{\e\to 0}\int_{E_\e}f=\int_{E_0}f.
\]
And the latter is true since $f$ is absolutely integrable on $E_0$, as can be shown with the same arguments as in section 3 of \cite{Gil}.

Therefore, it is enough to prove the second equality of the statement  in the particular case  $S=K\times (0,\infty)\subset\mathbb H^4$. In this case, by the Crofton formula
\[
\int_{\mathcal L_2^+}\#(\ell\cap S_\e)d\ell=\frac{4\pi^2}{3}A(S_\e)=\frac{4\pi^2}{3}L(K)\int_\e^\infty \frac{1}{t^2}dt =\frac{4\pi^2}{3\e}L(K)\]
By \eqref{measuredisks} we have
\[
 \int_{\mathcal L_{2,\e}^+}\#(\ell\cap S)d\ell=\frac{2\pi^2}{\e} L(K).
\]
Hence, all the limits in the statement vanish trivially for $S=K\times (0,\infty)$.
\end{proof}

By the previous proposition equation \eqref{gb1} becomes
\[
 \int_S (\kappa_i+1)dS=2\pi\chi(S)+\lim_{\e\to 0}\left(\frac{3}{4\pi^2}\int_{\mathcal L_{2,\e}^+}\#(\ell\cap S)d\ell-\frac{3}{2\e}L(K)\right).
\]
Combining this with Proposition \ref{renocirc} we get the following.

\begin{corollary}\label{gb} For any surface $S\subset\mathbb H^4$ with cone-like ends
\[
\frac1\pi\int_{N^1S}\kappa(x,e)dedS= \int_S (\kappa_i(x)+1)dS=2\pi\chi(S)+\frac{3}{4\pi^2}\int_{\mathcal L_{2}^+}(\#(\ell\cap S)-\lambda^2(\ell,K))d\ell-\frac{2}{\pi}E(K).
\]
In particular, $E(K)$ is invariant under M\"obius transformations.
\end{corollary}

A direct proof of the {M\"obius} invariance of $E(K)$ is given in subsection \ref{alternative.invariance}.

\subsection{Expressions via parallel curves}

Here we show the following
\begin{proposition}\label{prop_cos_cos_K_Kepsilon_space} 
Suppose the curvature of $K$ never vanishes. 
Let $K_{\e}$ be an $\e$-parallel curve given by $K_{\e}=\{x+\e\, n(x)\ | x\in K\}$, where $n$ is the unit principal normal vector to $K$.%\gil{ I modified the description of $K_\e$.} 
Then
\begin{equation}\label{f_cos_cos_K_Kepsilon_space}
E(K)=\lim_{\e\to0}\left(\frac{\pi}{4\e}L(K)+E(K,K_\e)\right)-\frac{\pi}{8}\int_K\kappa(p)dp,
\end{equation}
where $\kappa$ is the curvature of $K$. 
\end{proposition}

\begin{remark}\rm
 The previous hypothesis is no loss of generality: performing a M\"obius transformation we can bring every space curve $K$ to a position $\widetilde K$ where it has non-vanishing curvature. Moreover this transformation can be taken arbitrarily close to the identity. Indeed,
Let $CT$ denote the curvature tube of $K$, namely, $CT=\cup_{p\in K}C_O(p)$, where $C_O(p)$ denotes the osculating circle to $K$ at $p$. 
Then, the image of $K$ by the composition of an inversion in a sphere with a sufficiently big radius $r$ whose center does not belong to the curvature tube and is of distance $r$ from $K$, and a reflection in a plane gives the desired $\widetilde{K}$ after a motion of $\mathbb R^3$. 
\end{remark}
{}
We will need the following estimate.
\begin{lemma}\label{lemma_K_delta_B_epsilon_dp_dot_dq} 
Let $K$ be a simple space curve of class $C^2$ with a non-vanishing curvature. 
Let $\e$ and $\delta$ be small positive numbers with $\delta\ll\e$. 
Then for any point $p$ in $K$ we have 
\begin{equation}\label{eq_estimate}
\int_{K_\delta\cap B_\e(p)}\frac{\vect{v}_{p}\cdot\overrightarrow{dq}}
{|q-p|^2}
=\frac\pi\delta-\frac2\e-\frac\pi2\kappa(p)+O(\e),
\end{equation}
where $\vect{v}_{p}$ is the unit tangent vector to $K$ at $p$. 
\end{lemma}

\begin{proof}
{The proof is similar to that of Proposition \ref{lem_V_Laurent_e}. } 
Suppose $K$ can be expressed as $K=f(\mathbb S^1)$ by a $C^2$-embedding $f$ which is parametrized by the arc-length  {and the point $p$ is given by $p=f(0)$}. 
Then $K_\delta$ is given by $K_\delta=f_\delta(\mathbb S^1)$, where $f_\delta=f+\delta\kappa^{-1}f''$. 
Assume {$f(0)=0$}. 
Note that $f'\cdot f'\equiv1$ and $f''\cdot f''=\kappa^2$ imply 
\[
f'\cdot f''=0, \, f'\cdot f'''=-\kappa^2, \, f'\cdot f^{(4)}=-3\kappa\kappa', \>\>\mbox{ and }\>\>f''\cdot f'''=\kappa\kappa'. 
\]

The numerator of the left hand side of \eqref{eq_estimate} can be estimated as 
\begin{equation}\label{estimate_numerator}
f'(0)\cdot f_\delta'(s)=(1-\kappa(0)\delta)-\kappa'(0)\delta s-\frac{\kappa^2(0)+O(1)\delta}2\,s^2+O(s^3)
\end{equation}
since direct computation shows 
\[\begin{array}{rcl}
f'(0)\cdot f_\delta'(0)&=&\displaystyle 1-\kappa(0)\delta,\\[2mm]
f'(0)\cdot f_\delta''(0)%&=&\displaystyle \delta f'(0)\cdot\left(-2\kappa'(0)\kappa^{-2}(0)f'''(0)+\kappa^{-1}(0)f^{(4)}(0)\right)\\[2mm]
&=&\displaystyle -\kappa'(0)\delta,\\[2mm]
f'(0)\cdot f_\delta'''(0)&=&\displaystyle -\kappa^2(0)+O(1)\delta.
\end{array}\]

On the other hand, the denominator of the left hand side of \eqref{eq_estimate} can be estimated as 
\begin{eqnarray}
f_\delta(s)\cdot f_\delta(s)&=&\displaystyle f(s)\cdot f(s)+2\frac\delta{\kappa(s)}f(s)\cdot f''(s)+\delta^2 \nonumber\\%[4mm]
&=&\displaystyle \delta^2+\left(s^2-\frac{\kappa^2(0)}{12}s^4+O(s^5)\right) 
+2\delta\left(\frac1{\kappa(0)}+O(s)\right)
\left(-\frac{\kappa^2(0)}2s^2+O(s^3)\right) \nonumber\\%[4mm]
&=&\displaystyle \delta^2+\left(1-\kappa(0)\delta\right)s^2+\delta O(s^3)+O(s^4) \nonumber\\%[4mm]
&=&\displaystyle \left(\delta^2+(1-\kappa(0)\delta)s^2\right)\left(1+O(1)s^2\right). \label{estimate_denominator}
\end{eqnarray}
{We remark that \eqref{estimate_denominator} follows from the one above just like \eqref{eq_p-w_square2} from \eqref{eq_p-w_square}. }

Let us denote $\kappa(0)$ and $\kappa'(0)$ simply by $\kappa$ and $\kappa'$ in what follows. 
Let $s_-<0$ and $s_+>0$ be parameters when $f_\delta(s)$  passes through $\partial B_\e(f(0))$. 
Then, just like \eqref{eq_p-w_square2} implies \eqref{eq_s_pm}, \eqref{estimate_denominator} implies 
\[
s_\pm=\pm\sqrt{\frac{\e^2-\delta^2}{1-\kappa\delta}}+O(\e^3). 
\]

Therefore, by \eqref{estimate_denominator} and \eqref{estimate_numerator}, the left hand side of \eqref{eq_estimate} can be estimated as 
\[\begin{array}{l}
\displaystyle \int_{-\sqrt{\frac{\e^2-\delta^2}{1-\kappa\delta}}\,+O(\e^3)}^{\sqrt{\frac{\e^2-\delta^2}{1-\kappa\delta}}\,+O(\e^3)}
\frac{\left\{(1-\kappa\delta)-\kappa'\delta s+O(1)s^2\right\}\left(1+O(1)s^2\right)}
{\delta^2+(1-\kappa\delta)s^2}\,ds \\[5mm]
=\displaystyle \int_{-\sqrt{\frac{\e^2-\delta^2}{1-\kappa\delta}}}^{\sqrt{\frac{\e^2-\delta^2}{1-\kappa\delta}}}\>
\frac{1-\kappa\delta}{\delta^2+(1-\kappa\delta)s^2}\,ds+O(\e)\\[4mm]
=\displaystyle \frac2\delta\sqrt{1-\kappa\delta}\,\arctan\left(\frac{\sqrt{\e^2-\delta^2}}\delta\right)+O(\e)\\[5mm]
=\displaystyle \frac2\delta\left(1-\frac\kappa2\delta\right)\left(\frac\pi2-\frac\delta\e\right)+O(\e),\\[4mm]
\end{array}\]
which coincides with the right hand side of \eqref{eq_estimate}. 
\end{proof}

We remark that the lemma holds when $K$ and $K_\delta$ are planar curves with $K=\partial\Omega$ and $K_\delta=\partial\Omega_\delta$. 

Proposition \ref{prop_cos_cos_K_Kepsilon_space} is immediate from {Proposition \ref{renocirc}, Corollary \ref{E_space_K_1_K2}, and} the following. 
\begin{proposition}\label{prop_E_delta_parallel}
Let $K$ be a simple space curve of class $C^2$. 
Assume $K$ has a non-vanishing curvature. 
Let $K_\delta$ be a $\delta$-parallel of $K$ in the principal normal direction. Then 
\[
\lim_{\e\to0}\left(\frac{L(K)}{2\e}-\frac14\int_{K\times K\setminus\Delta_\e}\frac{\overrightarrow{dp}\cdot\overrightarrow{dq}}{|q-p|^2}\right)=\lim_{\delta\to0}\left(\frac\pi{4\delta}L(K)-\frac14\int_{K\times K_\delta}\frac{\overrightarrow{dp}\cdot\overrightarrow{dq}}{|q-p|^2}\right)
-\frac\pi8\int_K\kappa(p)dp.
\]
\end{proposition}

\begin{proof}
We first fix $\e$ so that $0<\e<1/\max_{p\in K}\kappa(p)$. 
Suppose $\delta<\e$. Then 
\[
-\frac14\int_{K\times K_\delta}\frac{\overrightarrow{dp}\cdot\overrightarrow{dq}}{|q-p|^2}
=-\frac14\int_{p\in K}\left(
\int_{q\in K_\delta\cap B_\e(p)}\frac{\vect v_p\cdot\overrightarrow{dq}}{|q-p|^2}+\int_{q\in K_\delta\cap {(B_\e(p))}^c}\frac{\vect v_p\cdot\overrightarrow{dq}}{|q-p|^2}
\right)dp.
\]
If $\delta\ll\e$ then 
\[
\lim_{\delta\to0}\int_{p\in K}\int_{q\in K_\delta\cap {(B_\e(p))}^c}\frac{\vect v_p\cdot\overrightarrow{dq}}{|q-p|^2}
=\int_{p\in K}\int_{q\in K\cap {(B_\e(p))}^c}\frac{\vect v_p\cdot\overrightarrow{dq}}{|q-p|^2}. 
\]
Therefore, by lemma \ref{lemma_K_delta_B_epsilon_dp_dot_dq} we have 
\[\begin{array}{rcl}
\displaystyle -\frac14\int_{K\times K_\delta}\frac{\overrightarrow{dp}\cdot\overrightarrow{dq}}{|q-p|^2}
&=&\displaystyle -\frac14\int_{p\in K}\left(\int_{K\setminus B_\e(p)}\frac{\vect v_p\cdot\overrightarrow{dq}}{|q-p|^2}
+\frac\pi\delta-\frac2\e-\frac\pi2\kappa(p)\right)dp+O(\e)\\[5mm]
&=&\displaystyle -\frac14\int_{K\times K\setminus\Delta_\e}\frac{\overrightarrow{dp}\cdot\overrightarrow{dq}}{|q-p|^2}-\frac\pi{4\delta}L(K)+\frac{L(K)}{2\e}+\frac\pi8\int_K\kappa(p)dp+O(\e),
\end{array}\]
which implies 
\[
\lim_{\delta\to0}\left(\frac\pi{4\delta}L(K)-\frac14\int_{K\times K_\delta}\frac{\overrightarrow{dp}\cdot\overrightarrow{dq}}{|q-p|^2}\right)
-\frac\pi8\int_K\kappa(p)dp
=\lim_{\e\to0}\left(\frac{L(K)}{2\e}-\frac14\int_{K\times K\setminus\Delta_\e}\frac{\overrightarrow{dp}\cdot\overrightarrow{dq}}{|q-p|^2}\right),
\]
which completes the proof. 
\end{proof}

%%%%%%%%%%%%%%%%%%%%%%%%%%%%%%%%%%%%%%%%%%%%%5
\section{M\"obius invariant expressions }\label{invariant.expressions}
For a compact simply connected domain $\Omega\subset\mathbb R^2$ with smooth boundary $K=\partial \Omega$, Theorem 1 in \cite{Gil} yields
\begin{equation}\label{tsint}
 E(\Omega)=\frac{\pi^2}{2}+\frac14\int_{K\times K}\theta\sin\theta \frac{dpdq}{|q-p|^2}
\end{equation}
where $\theta$ is the oriented angle at $p$ between $K$ (positively oriented), and the circle through $p$ and $q$ that is positively tangent to $K$ at $q$. More precisely, $\theta\in\mathbb R$  is the unique continuous determination of this angle defined on $K\times K$ that vanishes on the diagonal. Note that, unlike the previous expressions we obtained,  the integrand in \eqref{tsint} is pointwise M\"obius invariant.

Next we generalize \eqref{tsint} to compact domains, not necessarily simply connected. By equation \eqref{eq_E_K1_cup_K2}
 it is enough to give analogous expressions for the mutual energy $E(\Omega_1,\Omega_2)$ of two disjoint simply connected domains $\Omega_1,\Omega_2$. To this end, we will work with the flag space defined in \eqref{}, taking $n=2$. By thinking of $\mathbb R^2$ as the ideal boundary of half-space model of $\mathbb H^3$,  each element $(w,z;\gamma)\in \mathcal F$ corresponds to a pair $(\ell,\wp)$ where $\ell\subset\mathbb H^3$ is a geodesic line contained in the geodesic plane $\wp\subset\mathbb H^3$. Let us choose (locally) an orthonormal frame $(o;e_1,e_2,e_3)$ with $ o\in\ell$, $e_1\in T_o\ell$, $e_3\bot T_o\wp$. It is easy to check that the 1-form $\omega_{23}=\langle \nabla e_2,e_3\rangle$ is independent of this choice. Hence it defines a global 1-form $\varphi$ on $\mathcal F$. By construction, $\varphi$ is M\"obius invariant, it vanishes on the fibers of $\pi_2$ and measures the oriented angle on the fibers of $\pi_1$. The interest of $\varphi$ comes from the fact that, by \eqref{} we have $d\varphi=2\pi_1^*\I(\omega_{cr})$.

\begin{lemma}\label{llum}
 Let $c(t)=(z(t),w(t);\gamma(t))$ be a curve in $\mathcal F$, such that $z(t)\equiv z$ is constant and the circles $\gamma(t)$ are all mutually tangent at $z$. Then $\varphi(c'(0))=0$. 
\end{lemma}
\begin{proof}
 The curve $c(t)$ corresponds to a family of pairs $(\ell(t),\wp(t))$ with $\ell(t)\subset\wp(t)\subset\mathbb H^3$. By hypothesis, these geodesics $\ell(t)$ have a common ideal endpoint $z$. By a M\"obius transformation we can send $z$ to infinity. This way, the geodesics $\ell(t)$ become vertical lines in the model. Morevoer, by hypothesis the totally geodesic planes $\wp(t)$ are mapped to a family of vertical parallel affine planes in the model. Then we can choose a moving frame $o(t),e_1(t),e_2(t),e_3(t)$ adapted to (the image of) $(\ell(t),\wp(t))$ as above and such that $o(t)=(o^1(t),o^2(t),1)$, and $e_1(t),e_2(t),e_3(t)$ are constant vectors, forming an orthonormal basis of $\mathbb R^3$. Then clearly $\varphi(c'(t))=\langle e_2'(t),e_3(t)\rangle=0$.
\end{proof}

\begin{proposition} 
Let  $p_i:\mathbb S^1\rightarrow K_i$ be  regular parametrizations, and let $\theta(s,t)= \theta(p_1(s),p_2(t))\in [0, 2\pi)$ be the oriented angle  between the circle positively tangent to $K_1$ at $p_1(s)$, and the circle positively tangent to $K_2$ at $p_2(t)$.  Then
  \[
E(\Omega_1,\Omega_2)=\frac{\pi^2}{2}-\frac{1}{8}\int_{\mathbb S^1\times \mathbb S^1}\frac{\partial \theta(s,t)}{\partial s} \frac{\partial \theta(s,t)}{\partial t} dtds.
\]
\end{proposition}
\begin{proof}
Let us pick up a point $q_i\in \Omega_i$ ($i=1,2$) in the interior of each region. We denote $\Omega_i^*=\Omega_i\setminus\{q_i\}$. For each region, we take an orientation preserving diffeomorphism
\[
 F_i\colon \mathbb S^1\times[0,1)\longrightarrow \Omega_i^* \qquad i=1,2,
\]
{such that $F_i(t,0)=p_i(t)$.}
The vector field $X_i=\partial F_i(x,t)/\partial x$ is defined on $\Omega_i^*$ and vanishes nowhere. Let us define a section {$s_1:\Omega_1^*\times \Omega_2^*\rightarrow \mathcal F$} such that $s_1(w,z)=(w,z;\xi)$ with $X_1(w)\in T_{w}\xi$. Similarly, we define $s_2$ on $\Omega_1^*\times \Omega_2^*$ so that $X_2(z)\in T_z\xi$ if $s_2(w,z)=(w,z;\xi)$. 
Let $\Omega_{i,\e}=\Omega_i\setminus B_\e(q_i)$. 
Then
 \[
  E(\Omega_1,\Omega_2)=\lim_{\e\to 0}\frac{1}{8}\int_{\Omega_{1,\e}\times \Omega_{2,\e}} d(s_1^*\varphi)\wedge d(s_2^*\varphi)
 \]
By Stokes,
\[
 \int_{\Omega_{1,\e}\times \Omega_{2,\e}} d(s_1^*\varphi)\wedge d(s_2^*\varphi)=\int_{{(\partial\Omega_{1,\e}\times \Omega_{2,\e})}\cup ({\Omega_{1,\e}\times \partial\Omega_{2,\e}})}s_1^*\varphi\wedge d(s_2^*\varphi).
\]
Integration on $\partial\Omega_{1,\e}\times \Omega_{2,\e}$ vanishes by Lemma \ref{llum}. Using $d(s_1^*\varphi\wedge s_2^*\varphi)=ds_1^*\varphi\wedge s_2^*\varphi-s_1^*\varphi\wedge d s_2^*\varphi$ we get
\[
 \int_{\Omega_{1,\e}\times \partial\Omega_{2,\e}}s_1^*\varphi\wedge d(s_2^*\varphi)=-\int_{\Omega_{1,\e}\times \partial\Omega_{2,\e}}d(s_1^*\varphi\wedge s_2^*\varphi)+
\int_{\Omega_{1,\e}\times\partial \Omega_{2,\e}}ds_1^*\varphi\wedge s_2^*\varphi
\]
and the latter integral vanishes again by Lemma \ref{llum}. Taking care of orientations we conclude
\[\begin{array}{rcl}

 E(\Omega_1,\Omega_2)&=&\displaystyle \int_{K_1\times K_2}s_1^*\varphi\wedge s_2^*\varphi\\[4mm]
&&\displaystyle -\lim_{\e \to 0}\left(\int_{(K_1\times \partial B_\e(q_2))\cup(\partial B_\e(q_1)\times K_2)}s_1^*\varphi\wedge s_2^*\varphi+\int_{\partial B_\e(q_1)\times \partial B_\e(q_2)}s_1^*\varphi\wedge s_2^*\varphi\right). \end{array}
\]
Clearly,
\[
 \int_{K_1\times K_2}s_1^*\varphi\wedge s_2^*\varphi=\int_{\mathbb S^1\times \mathbb S^1}\frac{\partial \theta(s,t)}{\partial s} \frac{\partial \theta(s,t)}{\partial t} dtds.
\]
Applying the latter to the pairs of curves  $(K_1, \partial B_\e(q_2)), (\partial B_\e(q_1),K_2)),(\partial B_\e(q_1) ,\partial B_\e(q_2))$, and taking limits gives the result.
\end{proof}

\begin{proposition} 
Assume $K_i=\partial \Omega_i$ is connected for $i=1,2$, and let $K_i^*=K_i\setminus\{p_i^0\}$  for some arbitrary point $p_i^0\in K_i$. Then
 \[
E(\Omega_1,\Omega_2)=\frac{1}{4}\int_{K_1^*\times K_2^*}\theta\sin\theta\frac{dp_1dp_2}{|p_2-p_1|^2}\,,
\]
where $\theta(p_1,p_2)$ is any continuous determination on $K_1^*\times K_2^*$ of the oriented angle  between the circle positively tangent to $K_1$ at $p_1$, and the circle positively tangent to $K_2$ at $p_2$. 
\end{proposition}
\begin{proof} 
Let us use the {notations from the previous proof, with the convention} $\mathbb S^1=[0,1]/\sim$, where $0\sim1$. 
We can assume $p_i^0=F_i(0,0)$. Let $\Omega_i^\prime=F_i\left((\mathbb S^1\setminus \{0\})\times [0,1)\right)$.
To simplify the notation we will identify $\Omega_i^\prime\equiv (0,1)\times [0,1)$. We will also write $(x,t)\equiv F_1(x,t)=w, (y,u)\equiv F_2(y,u)=z$.
 Since $\Omega_1^\prime\times\Omega_2^\prime$ is (homotopically) contractible, the restriction $\mathcal F|_{\Omega_1^\prime\times\Omega_2^\prime}$ is a trivial bundle (i.e. there exists a bundle isomorphism $\tau:\mathcal F|_{\Omega_1^\prime\times\Omega_2^\prime}\rightarrow\Omega_1^\prime\times\Omega_2^\prime\times\mathbb S^1$). Moreover the row in the diagram below lifts
\[
\begin{CD}@.@. \Omega_1^\prime\times\Omega_2^\prime\times \mathbb R\\
@. @. @VVV\\
\Omega_1^\prime\times\Omega_2^\prime @>{s_j} >> \mathcal F|_{\Omega_1^\prime\times\Omega_2^\prime}@>\tau>\cong> \Omega_1^\prime\times\Omega_2^\prime\times\mathbb S^1.\\
\end{CD}
\]
i.e., there exist $f_j\colon \Omega_1^\prime\times\Omega_2^\prime\rightarrow \mathbb R$ such that $(w,z;\exp(if_j(w,z))=\tau(s_j(w,z))$.

Let now $\rho\colon\mathbb R\rightarrow \mathbb R$ be a $\mathcal C^\infty$ monotone function such that $\rho(x)=0$ for $x\leq 0$ and $\rho(x)=1$ for $x\geq 1$. Given $\e>0$ we define 
\[
h_\e(x,t,y,u)=\rho(u/\e)f_1(x,t,y,u)+\rho(t/\e)f_2(x,t,y,u).
\]
Then $s_\e(x,t,y,u)=\tau^{-1}(w,z; \exp(ih_\e(x,t,y,u))$ defines a section $s_\e$ of $\pi$ over $\Omega_1^\prime\times\Omega_2^\prime$ which we identified to $(0,1)\times [0,1)\times(0,1)\times[0,1)$. Hence we have $s_\e^*\varphi$ defined on $(0,1)\times [0,1)\times(0,1)\times[0,1)$. In fact, it extends to $\mathbb S^1\times [0,1)\times\mathbb S^1\times [0,1)$.  Next we take a small $\delta>0$ and we apply Stokes thorem to the manifold $U_\delta=\mathbb S^1\times[0,1-\delta]\times\mathbb S^1\times[0,1-\delta]$:
\begin{multline*}
 4 \int_{U_\delta} \frac{dzdw}{|z-w|^4}=\int_{\partial U_\delta}s_\e^*\varphi\wedge\I\omega_{cr}=\int_{\{t=1-\delta\}\cup\{u=1-\delta\} \cup \{t=0\}\cup\{u=0\}}s_\e^*\varphi\wedge\I \omega_{cr}.
\end{multline*}

The norm  $\|s_\e^*\varphi\|_\infty$ is bounded for a fixed $\e>0$. 
Besides,   $\|i_{\partial/\partial x}(\I\omega_{cr})\|=O(\delta)$  for $t=1-\delta$ or $u=1-\delta$. Hence,
 \[
  \lim_{\delta\rightarrow 0}\int_{\{t=1-\delta\}}s_\e^*\varphi\wedge\I\omega_{cr}=0
 \]
\[
  \lim_{\delta\rightarrow 0}\int_{\{u=1-\delta\}}s_\e^*\varphi\wedge\I\omega_{cr}=0
 \]
 On the other hand
 \begin{equation}\label{twoequations}
    \int_{\{t=0,u>\e\}}s_\e^*\varphi\wedge\I\omega_{cr}0,\qquad=\int_{\{u=0,t>\e\}}s_\e^*\varphi\wedge\I\omega_{cr}=0
\end{equation}
Indeed, for $t=0,u>\e$, $s_\e=s_1$. In this case,  by Lemma \ref{llum}
\[
 s_\e^*\varphi\frac{\partial}{\partial y}=s_\e^*\varphi\frac{\partial}{\partial u}=0.
\]
Hence, 
\[
 (s_\e^*\varphi\wedge\I\omega_{cr})\left(\frac{\partial}{\partial x},\frac{\partial}{\partial y},\frac{\partial}{\partial u}\right)=0.
\]
This shows the first equation in \eqref{twoequations}. The second one follows by symmetry. We have shown so far that
\begin{equation}\label{almost}
E(\Omega_1,\Omega_2)=\frac{1}{4}\int_{\{0<x,y<1,t=0,0<u<\e\}\cup\{0<x,y<1,u=0,0<t<\e\}}s_\e^*\varphi\wedge\I \omega_{cr}.
\end{equation}

To compute the latter integral we take the limit as $\e$ goes down to $0$.
 \[\begin{array}{l}
\displaystyle \lim_{\e\rightarrow 0}\int_{\{0<x,y<1,t=0,0<u<\e\}}s_\e^*\varphi\wedge\I\omega_{cr}\\[4mm]
\displaystyle =\lim_{\e\rightarrow 0}\int_0^1\int_0^1\int_0^\e\left(s_\e^*\varphi\wedge\I\omega_{cr}\right)_{(x,0,y,u)}\left(\frac{\partial}{\partial x},\frac{\partial}{\partial y},\frac{\partial}{\partial u}\right)dudydx\\[4mm]
=\displaystyle \lim_{\e\rightarrow 0}\int_0^1\int_0^1\int_0^1\e\left(s_\e^*\varphi\wedge\I\omega_{cr}\right)_{(x,0,y,\e v)}\left(\frac{\partial}{\partial x},\frac{\partial}{\partial y},\frac{\partial}{\partial u}\right)dvdydx. \end{array}
\]
Since the norm of $\e{s_\e}_{\ast}$ is uniformly bounded, we may apply Lebesgue's dominated convergence theorem to put the limit inside the integral. Since $s_\e^*\varphi(\partial/\partial x), s_\e^*\varphi(\partial/\partial y)$ are uniformly bounded,
\[
 \int_0^1\int_0^1\int_0^1\lim_{\e\rightarrow 0}\e\left(s_\e^*\varphi\wedge\I\omega_{cr}\right)_{(x,0,y,\e v)}\left(\frac{\partial}{\partial x},\frac{\partial}{\partial y},\frac{\partial}{\partial u}\right)dvdydx\]
\[= \int_0^1\int_0^1\left(\int_0^1\lim_{\e\rightarrow 0}\e s_\e^*\varphi\left(\frac{\partial}{\partial u}\right)dv\right)\I\omega_{cr}\left(\frac{\partial}{\partial x},\frac{\partial}{\partial y}\right)dydx
\]
Finally, by continuity, and since $\lim_{\e\to0}\e{s_\e}_{\ast}(\frac{\partial}{\partial u})$  is tangent to the fibers of $\pi_1$,
\[
 \lim_{\e\rightarrow 0}\e \varphi\left({s_\e}_{\ast}\left(\frac{\partial}{\partial u}\right)\right)=\varphi\left(\lim_{\e\rightarrow 0}\e {s_\e}_{\ast}\left(\frac{\partial}{\partial u}\right)\right)=\lim_{\e\rightarrow 0}\frac{\partial h_\e}{\partial u}{(x,0,y,\e v)}=\rho'(v)f_1(x,0,y,0).
\]
Therefore
 \[
  \lim_{\e\rightarrow 0}\int_{\{(x,y)\}}\int_{\{0<u<\e\}}\left(s_\e^*\varphi\wedge\I\omega_{cr}\right)_{(x,0,y,u)}=\int_{\{(x,y)\}}f_1(x,0,y,0)\I\omega_{cr}.
 \]
Similarly, 
\[
  \lim_{\e\rightarrow 0}\int_{\{(x,y)\}}\int_{\{0<t<\e\}}\left(s_\e^*\varphi\wedge\I\omega_{cr}\right)_{(x,t,y,0)}=-\int_{\{(x,y)\}}f_2(x,0,y,0)\I\omega_{cr}.
 \]
This, together with \eqref{almost}, completes the proof since $\theta=f_2-f_1$, and 
\[
 \I\omega_{cr}=\sin\theta\frac{ dx_1dx_2}{|x_2-x_1|^2}.
\]
\end{proof}

\section{Alternative proofs} 
\subsection{Alternative proof of proposition \ref{relation}}\label{subsec_alt_pf_prop_gil_E'}
We give a direct proof of 
\[\lim_{\e\rightarrow 0}\left(-\int_{\Omega\times\Omega^c\setminus\Delta_\e}\frac{\da w\da z}{|z-w|^4}+\frac{2}{\e}L(K)\right)
=\lim_{\delta\to0}\left(\int_{\Omega_{\delta}}V_\Omega(w)\,\da w+\frac{\pi}{4\delta}L(K)\right)
-\frac{\pi^2}4\chi(\Omega)\]
when $\partial\Omega$ is connected, i.e. when $\chi(\Omega)=1$.

Suppose $0<\delta\le\e$. 
Note that 
\[
\Omega\times\Omega^c\setminus\Delta_\e
=\left((\Omega\setminus\Omega_\delta)\cup\Omega_\delta\right)\times\Omega^c\setminus\Delta_\e
=(\Omega\setminus\Omega_\delta)\times\Omega^c\setminus\Delta_\e\cup\left(\Omega_\delta\times\Omega^c\setminus D\right),
\]
where
\[
D=\{(w,z)\,|\,w\in\Omega_\delta\setminus\Omega_\e, \>z\in\Omega^c\cap B_\e(w)\}.
\]
Since $(\Omega\setminus\Omega_\delta)\times\Omega^c\setminus\Delta_\e\subset(\Omega\setminus\Omega_\delta)\times\mathbb R^2\setminus\Delta_\e$ we have 
\[
\lim_{\delta\to0}\int_{(\Omega\setminus\Omega_\delta)\times\Omega^c\setminus\Delta_\e}\frac{\da w\da z}{|z-w|^4}
\le\lim_{\delta\to0}\,\frac{\pi^2}{\e^2}\cdot A(\Omega\setminus\Omega_\delta)=0,
\]
which implies 
\[
\begin{array}{rcl}
\displaystyle -\int_{\Omega\times\Omega^c\setminus\Delta_\e}\frac{\da w\da z}{|z-w|^4}
&=&\displaystyle \lim_{\delta\to0}\left(-\int_{\Omega_\delta\times\Omega^c}\frac{\da w\da z}{|z-w|^4}+\int_D\frac{\da w\da z}{|z-w|^4}\right)\\[4mm]
&=&\displaystyle \lim_{\delta\to0}\left(\int_{\Omega_{\delta}}V_\Omega(w)\,\da w+\int_{\Omega_\delta\setminus\Omega_\e}V_{\Omega^c\cap B_\e(w)}(w)\da w\right).
\end{array}
\]
Therefore, we have only to prove 
\begin{lemma}
Suppose $\delta\ll\e^2$. 
Then 
\[
\int_{\Omega_\delta\setminus\Omega_\e}V_{\Omega^c\cap B_\e(w)}(w)\da w
=\frac{\pi}{4\delta}L(K)-\frac{2}{\e}L(K)-\frac{\pi^2}4+O(\e). 
\]
\end{lemma}
\begin{proof}
Suppose $w\in \Omega_\delta\setminus\Omega_\e$. 
Put $t=\textrm{dist}(w,K)$ and let $\kappa(s)$ be the curvature of $K$ at the closest point to $w$. 
Then \eqref{eq_key_estimate1} and \eqref{eq_key_estimate2} in the proof of Proposition \ref{lem_V_Laurent_e} imply that 
\begin{eqnarray*}
&&V(w,\Omega^c\cap B_\e(w))
=-\frac12\int_{\partial\Omega^c\cap B_\e(w)}\frac{\det\big(p-w, \overrightarrow{dp}\,\big)}{|p-w|^4}
-\frac12\int_{\Omega^c\cap \partial B_\e(w)}\frac{\det\big(p-w, \overrightarrow{dp}\,\big)}{|p-w|^4}\\
&&
=
\left(\frac{1+\frac{\kappa(s) t}2}{2t^2\sqrt{1-\kappa(s) t}}-\frac1{\e^2}\right)\arccos\left(\frac t\e\right)
+\frac{\left(1-\frac{\kappa(s) t}2\right)\sqrt{\e^2- t^2}}{2 t\e^2\sqrt{1-\kappa(s) t}}
-\frac{\kappa(s)\sqrt{\e^2- t^2}}{2\e^2}+O(\e).
\end{eqnarray*}
Let the right hand side above be denoted by $V_{loc}(s,t)$. 
Then we have
\[
\int_{\Omega_\delta\setminus\Omega_\e}V_{\Omega^c\cap B_\e(w)}(w)\da w
=\int_0^{L(K)}\!\!\int_\delta^\e\left(1-\kappa(s)t\right)V_{loc}(s,t)\,dt\,ds+O(\e).
\]
If we fix $s$ and put $\kappa=\kappa(s)$, then 
\begin{equation}\label{f_last}
\begin{array}{rcl}
\displaystyle \int_\delta^\e\left(1-\kappa(s)t\right)V_{loc}(s,t)\,dt 
&=&\displaystyle \int_{\frac\delta\e}^1\left[
\left(\frac{(1+\frac{\kappa\e}2t)\sqrt{1-\kappa\e t}}{2\e t^2}-\frac{1-\kappa\e t}\e\right)\arccos t\right.\\[5mm]
&&\displaystyle \left. \hspace{1cm}
+\frac{(1-\frac{\kappa\e}2t)\sqrt{1-\kappa\e t}\,\sqrt{1-t^2}}{2\e t}
-\frac\kappa2(1-\kappa\e t)\sqrt{1-t^2}\,
\right]dt.
\end{array}
\end{equation}

First assume $\kappa=0$. 
Then the above is given by 
\[
\int_{\frac\delta\e}^1\left[\left(\frac1{2\e t^2}-\frac1\e\right)\arccos t+\frac{\sqrt{1-t^2}}{2\e t}\,\right]dt
=\frac{\pi}{4\delta}-\frac2{\e}+\frac\delta{\e^2}\arccos\left(\frac\delta\e\right). 
\]

The contribution of $\kappa$ being non-zero can be estimated by expansion to a series in $\kappa$. 
Let $\varphi(\kappa,t)$ denote the integrand of the right hand side of \eqref{f_last}. 
Then 
\[
\begin{array}{rcl}
\displaystyle \lim_{\delta\to0} \int_{\frac\delta\e}^1\left. \frac{\partial}{\partial\kappa}\varphi(\kappa,t)\right|_{\kappa=0}\,dt
&=&\displaystyle \int_0^1\left(t\arccos t-\sqrt{1-t^2}\,\right) dt=-\frac{\pi}8,\\[4mm]
\displaystyle \lim_{\delta\to0} \int_{\frac\delta\e}^1\left. \frac{\partial^2}{\partial\kappa^2}\varphi(\kappa,t)\right|_{\kappa=0}\,dt
&=&\displaystyle \e\int_0^1\left(-\frac{3}8\arccos t+\frac{9}8t\sqrt{1-t^2}\,\right) dt=O(\e).
\end{array}
\]
It follows that 
\[
\int_\delta^\e\left(1-\kappa(s)t\right)V_{loc}(s,t)\,dt
=\frac{\pi}{4\delta}-\frac2{\e}-\frac{\pi}8\kappa(s)+O\left(\frac\delta{\e^2}\right)+O(\e),
\]
which completes the proof. 
\end{proof}

\subsection{Direct proof of Proposition \ref{renocirc}}
We show directly that
 \[
  \lim_{\e\to 0}\int_{S_\e(1,3)}\lambda^2(\gamma,K)d\gamma-\frac{2\pi^2}{\e}L(K)=\frac{8\pi}{3}\left(\lim_{\e\to 0}\int_{K\times K\setminus\Delta_\e}\cos\theta_p\cos\theta_q\frac{dpdq}{|q-p|^2}-\frac{2L(K)}{\e}\right).
 \]
where $\theta_p,\theta_q\in[0,\pi]$ are the angles at $p,q$ between the tangent vector of $K$ and the vector $q-p$.
\begin{proof}
 The invariant measure of circles is given by
\[
 d\gamma=\frac{1}{r^4}drdcdn
\]
where $r$ denotes the radius, $dc$ is the volume element relative to the center, and $dn$ is the area element in $\mathbb S^2$ relative to the unit vector $n$ normal to the plane containing $\gamma$, and determined by the orientation. Hence we must study
\[
 I_\e(K)=\int_{S_\e(1,3)}\lambda^2(\gamma,K)d\gamma=\int_\e^\infty \frac{1}{r^4}f(r,K)dr
\]
where
\[
 f(r,K)=\int_{\mathbb R^3}\int_{\mathbb S^2}\lambda^2(\gamma,K)dndc.
\]
The latter is the (isometry invariant) measure (with multiplicity) of radius $r$ circles linked with $K$. By the results of des Cloizeaux-Ball (cf. \cite{descloizeaux})
\[
 f(r,K)=\pi\int_0^{2r}\mathcal A_{K}(s)\sqrt{4r^2-s^2}ds
\]
with 
\[
 \mathcal A_K(s)=\frac{1}{s}\int_{K\times K\cap \Delta_s}\overrightarrow{dp}\cdot \overrightarrow{dq}=\int_{K}\sum \cos\theta_p\cdot\mathrm{sign}(\cos\theta_q) dp.
\]
where $\overrightarrow{dp}\cdot\overrightarrow{dq}=\sum_{i=1}^3 dp^i\wedge dq^i$, and the sum runs over the points $q\in K$ at distance $s$ from $p$. Therefore, 
\[
 I_\e(K)=\pi\int_\e^\infty \frac{1}{r^4}\int_0^{2r}\mathcal A_K(s)\sqrt{4r^2-s^2}dsdr
\]
Using $\mathcal A_K(s)=2L(K)+O(s^2)$, and Fubini's theorem  we get
\[
I_\e(K) =\pi\int_\e^\infty \frac{1}{r^4}\int_0^{2\e} (2L(K)+O(s^2))\sqrt{4r^2-s^2}dsdr+\pi\int_{2\e}^\infty\mathcal A_K(s)\int_{s/2}^\infty\frac{\sqrt{4r^2-s^2}}{r^4}drds
\]
Computing the integrals, and using $\sqrt{4r^2-s^2}\leq 2r$ gives 
\[I_\e(K)=\pi\int_\e^\infty \frac{4}{r^4}L(K)(\e\sqrt{r^2-e^2}+r^2\arctan(\frac{\e}{\sqrt{r^2-\e^2}})dr+O(\e^3)\int_\e^\infty\frac{1}{r^3}dr+\pi\int_{2\e}^\infty\frac{8}{3s^2}\mathcal A_K(s)ds 
\]
\[
 =\pi(-\frac{8}{3\e}+\frac{2\pi}{\e})L(K)+O(\e)+\frac{8\pi}{3}\int_{2\e}^\infty \frac{A_K(s)}{s^2}ds
\]
To finish the proof we just need to remark that 
\[
 \int_{2\e}^\infty \frac{\mathcal A_K(s)}{s^2}ds=\int_{K\times K\setminus\Delta_{2\e}}\cos\theta_p\cos\theta_q\frac{dpdq}{|q-p|^2}.
\]
\end{proof}

\subsection{Alternative proof of the M\"obius invariance of $E$ for single space curves}\label{alternative.invariance}
We give {a sketch of an }alternative proof of the M\"obius invariance of $E$ for space curves using the M\"obius invariance of $E$ for pairs of curves $\displaystyle E(K_1,K_2)=-\frac14\int_{K_1\times K_2}\frac{\overrightarrow{dp}\cdot\overrightarrow{dq}}{|q-p|^2}$ and the renormalization formula 
\[
E(K)=\displaystyle \lim_{\delta\to0}\left(\frac{\pi}{4\delta}L(K)-\frac14\int_{K\times K_{\delta}} \frac{\overrightarrow{dp}\cdot\overrightarrow{dq}}{|q-p|^2}\right)-\frac{\pi}{8}\int_K\kappa(p)dp.
\]

It is enough to show that $E$ is invariant under an inversion $I$ in a sphere with center $C$ and radius $1$. 
Let us express the image of $I$ by putting tilde above, namely, we put 
$\wt{K}=I(K), \wt{K_\delta}=I(K_\delta), \wt{n}=\displaystyle \frac{I_{\ast}(n)}{|I_{\ast}(n)|}$, and $\wt{p}=I(p)$ for $p\in K$. 
Let $\wt{\kappa}$ be the curvature of $\wt{K}$, $n_{\wt{K}}$ the unit principal normal vector of $\wt{K}$, and $\beta$ the angle between $\wt{n}$ and $n_{\wt{K}}$. %=\beta(p)=\beta(\wt{p})
Assume that the curvatures of $K$ and $I(K)$ do not vanish anywhere. 
\begin{lemma}\label{lemma_estimate_ball_beta}
Suppose $0<\delta\ll\e$. 
Let $\wt p$ be a point on $\wt K$. 
Then 
\begin{equation}\label{eq_estimate_w_beta}
\int_{\left(\wt K+\delta\wt{{n}}\right)\cap B_\e(\wt p)}\frac{{v}_{\wt p}\cdot\overrightarrow{dq}}
{|\wt p-q|^2}
=\frac\pi\delta-\frac2\e-\frac\pi2\wt{\kappa}(\wt p)\cos\beta(\wt p) +O(\e).
\end{equation}
\end{lemma}
\begin{proof}
In order to estimate the integral up to $O(\e)$, we only need up to quadratic terms of the curve, as in the proof of lemma \ref{lemma_K_delta_B_epsilon_dp_dot_dq}. 
Therefore, we may assume that $\wt K$ is a subarc of a circle near the point $\wt p$. 
We may further assume that the angle $\beta$ is constant near $\wt p$ since the contribution of the variation of $\beta$ tends to $0$ as $\delta$ goes down to $0$. 

Put 
\[R=\frac1{\wt{\kappa}}, \wt p=(R,0,0), \>\mbox{ and }\> q=((R-\delta\cos\beta)\cos\theta, (R-\delta\cos\beta)\sin\theta, \delta\sin\beta)\] 
for $-\theta_0\le\theta\le\theta_0$, where $\theta_0$ is given by 
$\displaystyle \sin\frac{\theta_0}2=\frac12\sqrt{\frac{\e^2-\delta^2}{R(R-\delta\cos\beta)}}$. 
Then the left hand side of \eqref{eq_estimate_w_beta} can be estimated as 
\[\begin{array}{l}
\displaystyle 2\int_0^{\theta_0}\frac{\cos\theta}{4R(R-\delta\cos\beta)(1-\cos\theta)+\delta^2}
\cdot(R-\delta\cos\beta)\,d\theta\\[4mm]
\sim \displaystyle \frac2\delta\sqrt{1-\frac{\delta\cos\beta}R}
\int_0^{\frac{\sqrt{\e^2-\delta^2}}\delta}
\frac{1+\sum_{k=1}^{\infty}O(1)\delta^{2k}t^{2k}}{1+t^2}\,dt \\[4mm]
= \displaystyle \frac2\delta\left(1-\frac{\cos\beta}2\wt \kappa\delta\right)
\left(\frac\pi2-\frac\delta\e\right)+O(\e)\\[4mm]
=\displaystyle \frac\pi\delta-\frac2\e-\frac\pi2\wt \kappa\cos\beta +O(\e).
\end{array}\]
\end{proof}
We remark that \eqref{eq_estimate_w_beta} coincides with \eqref{eq_estimate} when $\beta=0$. 
\begin{corollary}
\[
E\big(\wt K\big)=\lim_{\delta\to0}\left(\frac\pi{4\delta}L\big(\wt K\big)-\frac14\int_{\wt K\times \left(\wt K+\delta\wt{{n}}\right)}\frac{\overrightarrow{d\wt p}\cdot\overrightarrow{dq}}{|\wt p-q|^2}\right)
-\frac\pi8\int_{\wt K}\cos\beta(\wt p)\wt \kappa(\wt p)\,d\wt p.
\]
\end{corollary}
\begin{lemma}\label{lemma_cos_beta_n}
As before, 
\begin{equation}\label{f_cos_beta_kappa}
\int_{\wt{K}}\cos\beta(\wt{p})\wt{\kappa}(\wt{p})d\wt{p}
=\int_K\kappa(p)dp+2\int_K\frac{(p-C)\mbox{\Large$\cdot$}n}{|C-p|^2}\,dp, 
\end{equation}
where $C$ is the center of the sphere of the inversion $I$. 
\end{lemma}
\begin{proof}
Assume that $K\cap B_\e(p)$ coincides with a subarc of a circle $\Gamma$ in a plane $\Pi_1$. 
As $\wt{K_\delta}$ belongs to the sphere $I(\Pi_1)$, the angle $\beta$ between $\wt{n}$ and $n_{\wt{K}}$ is equal to the angle between $I(\Pi_1)$ and the $2$-plane through $I(\Gamma)$, which we denote by $\Pi_2$. 
As the inversion is conformal, the angle $\beta$ is equal to the angle between $\Pi$ and $I(\Pi_2)$. 
Note that $I(\Pi_2)$ is the sphere that passes through both the circle $\Gamma$ and the point $C$. 

Let $O$ be the center of $\Gamma$, $h$ the distance between $C$ and $\Pi_1$, $\rho$ the distance between the center of the sphere $I(\Pi_2)$ and $\Pi_1$, and $r$ the radius of $I(\Pi_2)$. 
Then $2h\rho=|C-O|^2-R^2$. 

Let $p_1, p_2$ be a pair of antipodal points on $\Gamma$ so that $p_1, p_2$, and the foot of a perpendicular to $\Pi_1$ through $C$ are collinear. 
Then the curvature $\wt{\kappa}(\tilde{p})$ of $\wt{K}$ at $\wt{p}$ is given by (cf. \cite{OH1} page 37)
\[\wt{\kappa}(\tilde{p})= \frac{2}{|I(p_1)-I(p_2)|}
=\frac{2|C-p_1||C-p_2|}{|p_1-p_2|}=|C-p_1||C-p_2|\,\kappa(p)=2hr\,\kappa(p). \]
Since $\cos\beta(p)=\rho/r$ 
\[\cos\beta(\wt{p})\wt{\kappa}(\wt{p})=2h\rho \kappa(p)=\left(|C-O|^2-R^2\right)\kappa(p).\]
Since 
\[
|C-O|^2-R^2=|C-p-Rn|^2-R^2=|C-p|^2-2R(C-p)\cdot n
\]
and $\displaystyle d\wt{p}=\frac{dp}{|C-p|^2}$ (cf. \cite{OH1} page 37), we have 
\[\cos\beta(\wt{p})\wt{\kappa}(\wt{p})\,d\wt{p}
=\frac{|C-O|^2-R^2}{|C-p|^2}\,\kappa(p)\,dp
=\kappa(p)\,dp+2\frac{(p-C)\mbox{\Large$\cdot$}n}{|C-p|^2}\,dp,
\]
which completes the proof. 
\end{proof}

\begin{lemma}\label{last_lem_for_thm_conf_inv_E} 
Suppose $\delta_0>0$ is smaller than the distance between $\wt{K}$ and $\wt{K_{\delta}}$. 
Then 
\begin{equation}\label{f_band}
\displaystyle 
\int_{\wt{K}\times \wt{K_{\delta}}} \frac{\overrightarrow{d\wt p}\cdot\overrightarrow{d\wt q}}{|\wt p-\wt q|^2}
-\int_{\wt{K}\times \left(\wt{K}+\mbox{\small $\delta_0$}{\wt{n}}\right)} 
\frac{\overrightarrow{d\wt p}\cdot\overrightarrow{dq}}{|\wt p-q|^2}
\displaystyle \sim
-\frac{\pi}{\delta_0}L\big(\wt{K}\big)
+\frac{\pi}{\delta}L(K)
+{\pi}\int_K\frac{(p-C)\mbox{\Large$\cdot$}n}{|C-p|^2}\,dp.
\end{equation}
\end{lemma}
\begin{proof}
Suppose $\wt{K_{\delta}}$ is approximated by $\wt{K_{\delta}}\approx \wt K+\wt \delta\,\wt{{n}}$ for sume function $\wt{\delta}:\wt K\to\mathbb R_{>0}$. 
Then 
\begin{eqnarray}
\displaystyle \frac{d\wt p}{\wt{\delta}(\wt p)}&=&\displaystyle \frac{|C-p|\,|C-(p+\delta n)|}{\delta}\cdot\frac{dp}{|C-p|^2}\nonumber\\%[4mm]
&=&\displaystyle \frac1\delta\cdot\frac{|C-(p+\delta n)|}{|C-p|}\,dp \nonumber\\%[4mm]
&=&\displaystyle \frac1\delta\left(1+\frac{(p-C)\cdot n}{|C-p|^2}\,\delta+O(\delta^2)\right)dp.\label{f_dtildep/tildedelta}
\end{eqnarray}
Let $\e>0$ satisfy $0<\delta_0<\wt{\delta}(\wt p)\ll\e$ $(\forall\wt p)$. 
We may assume that $\wt \delta$ is constant in $B_\e(\wt p)\cap\wt K$ for any $\wt p$, i.e. the contribution of the variation of $\wt \delta$ can be neglected. 

By a similar argument as in the proof of proposition \ref{prop_E_delta_parallel}, lemma \ref{lemma_estimate_ball_beta}, and \eqref{f_dtildep/tildedelta}, the left hand side of \eqref{f_band} can be estimated as 
\[\begin{array}{l}
\displaystyle \int_{\wt p\in\wt K}\left[
\left(
\int_{\wt{K_\delta}\cap B_\e(\wt p)}
-\int_{\left(\wt{K}+\delta_0\wt{n}\right)\,\cap B_\e(\wt p)}
\right)
\frac{ v_p\cdot\overrightarrow{dq}}{|q-p|^2}
+\left(
\int_{\wt{K_\delta}\cap (B_\e(\wt p))^c}
-\int_{\left(\wt{K}+\delta_0\wt{n}\right)\,\cap (B_\e(\wt p))^c}
\right)
\frac{ v_p\cdot\overrightarrow{dq}}{|q-p|^2}
\right]d\wt p \\[6mm]
=\displaystyle \pi\int_{\wt K}\frac{d\wt p}{\wt{\delta}(\wt p)}-\pi\int_{\wt K}\frac{d\wt p}{\delta_0}+O(\e),
\end{array}\]
which tends to the right hand side of \eqref{f_band} by \eqref{f_dtildep/tildedelta}. 
\end{proof}

Let us now proceed to an alternative proof of the M\"obius invariance of $E(K)$. 
By lemmas \ref{last_lem_for_thm_conf_inv_E} and \ref{lemma_cos_beta_n}, and the M\"obius invariance $E(K, K_\delta)=E\big(\wt K, \wt{K_\delta}\big)$, we have
\[\begin{array}{rcl}
E\big(\wt K\big)&=&\displaystyle \lim_{\delta_0\to0}\left(\frac\pi{4\delta_0}L\big(\wt K\big)-\frac14\int_{\wt K\times \left(\wt K+\delta_0\wt{{n}}\right)}\frac{\overrightarrow{d\wt p}\cdot\overrightarrow{dq}}{|\wt p-q|^2}\right)
-\frac\pi8\int_{\wt K}\cos\beta(\wt p)\wt \kappa(\wt p)\,d\wt p\\[5mm]
&=&\displaystyle \lim_{\delta\to0}\left(\frac\pi{4\delta}L(K)-\frac14\int_{\wt K\times \wt{K_\delta}}\frac{\overrightarrow{d\wt p}\cdot\overrightarrow{d\wt q}}{|\wt p-\wt q|^2}
+\frac{\pi}4\int_K\frac{(p-C)\mbox{\Large$\cdot$}n}{|C-p|^2}\,dp
\right)
-\frac\pi8\left(\int_{K}\kappa(p)\,dp+2\int_K\frac{(p-C)\mbox{\Large$\cdot$}n}{|C-p|^2}\,dp\right)\\[5mm]
&=&\displaystyle \lim_{\delta\to0}\left(\frac{\pi}{4\delta}L(K)-\frac14\int_{K\times K_{\delta}} \frac{\overrightarrow{dp}\cdot\overrightarrow{dq}}{|q-p|^2}\right)-\frac{\pi}{8}\int_K\kappa(p)dp \\[5mm]
&=&\displaystyle E(K),
\end{array}\]
which completes the proof.

%%%%%%%%%%%%%%%%%%%%%%%%%%%%%%%%%%%%%%%%%%%%%%%%%%%%%%%%%%%%%%%%%%%%%%%%%%%%%%%%%%%%%%%%%%%


\begin{thebibliography}{E} 
\bibitem[AS]{auckly} D.~Auckly, L.~Sadun, {\em A family of M\"obius invariant 2-knot energies}, Geometric Topology (Athens, GA, 1993), Studies in Advanced Math, AMS, 1997. 
 
%
\bibitem[AM]{alexakis}S.~Alexakis, R.~Mazzeo, {\em Renormalized area and properly embedded minimal surfaces in hyperbolic 3-manifolds}, arXiv:0802.2250v2.
%
\bibitem[BFHW]{bryson}S.~Bryson, M.~Freedman, Z.-X.~He, Z.~Wang, {\em
M\"obius invariance of knot energy},
Bull. Amer. Math. Soc. (N.S.) 28 (1993), no. 1, 99--103. 
%
\bibitem[BP]{banchoff.pohl} T.F.Banchoff, W.Pohl
\newblock{\em A generalization of the isoperimetric inequality.}
\newblock J. Diff. Geom., {\bf 6} (1971), 175--192.
%
\bibitem[BW]{BW} T. Banchoff and J.~H. White, 
{\em The behavior of the total twist and self-linking number of a closed 
space curve under inversions,} Math. Scand. {\bf 36} (1975), 254\,--\,262. 
%
\bibitem[BS]{Bu-Si}G.~Buck and J.~Simon, 
{\em Thickness and crossing number of knots,}  Topology Appl. {\bf 91} (1999), 245\,--\,257.
%
\bibitem[dCB]{descloizeaux}J.~des Cloizeaux, R.~Ball, {\em Rigid curves at random position and linking number}, Commun. Math. Phys. {\bf 80} (1981), 543--553.
%
%\bibitem[KS]{KS} R.~Kusner, and J.M.~Sullivan, {\em M\"obius invariant knot energies,} Ideal Knots. A. Stasiak, V. Katrich, L.~H. Kauffman eds., World Scientific, Singapore (1998), pp. 315\,--\,352.
%
\bibitem[LO]{La-OH}R.~Langevin and J.~O'Hara, {\em Conformally invariant energies of knots,} J. Institut Math. Jussieu 4 (2005), 219-280. 
%
\bibitem[O1]{jun}J.~O'Hara, {\em Energy of a knot}, Topology  30  (1991),  no. 2, 241--247.
%
\bibitem[O2]{OH1}J.~O'Hara, {\em Energy of knots and conformal geometry}. Series on Knots and Everything Vol. 33, World Scientific, Singapore, xiv + 288 pages. 
%
\bibitem[O3]{OH2}J.~O'Hara, {\em Renormalization of $r^{\bullet}$-potentials and generalization of dual volumes and centers}, {arXiv:1008.2731}. 
%
\bibitem[Poh]{pohl}W.Pohl, 
\newblock{\em Some integral formulas for space curves and their generalization.}
\newblock Amer. J. Math., {\bf 90}, n.4 (1968), 1321--1345.
%
\bibitem[San]{santalo}L.A.~Santal\'o, {\em Integral geometry and geometric probability}, Cambridge University Press, 2004.
%
\bibitem[So]{Gil}G.~Solanes, {\em Total curvature of complete surfaces in hyperbolic space}, Advances in Mathematics 225 (2010),  805--825.
%
\end{thebibliography}
\end{document}